\documentclass{article}

\usepackage[numbers]{natbib}

\usepackage{hyperref}
\hypersetup{
    colorlinks=true,
     allcolors =blue,
    pdftitle={Real and complex stability radii for a class of transport networks},
    pdfpagemode=FullScreen,
    pdfauthor={A. Hastir, B. Jacob, and H. Wagener}
    }

\usepackage[english]{babel}
\usepackage[utf8]{inputenc}
\usepackage{float}
\usepackage{amsmath}
\usepackage{amssymb}
\usepackage{esint}
\usepackage{amsthm}

\usepackage[a4paper,margin=2cm]{geometry}

\usepackage[toc,page]{appendix}

\usepackage{mathrsfs}
\usepackage{bm}

\newtheorem{theorem}{Theorem}[section]
\newtheorem{corollary}[theorem]{Corollary}
\newtheorem{lemma}[theorem]{Lemma}
\newtheorem{proposition}[theorem]{Proposition}
\theoremstyle{definition}
\newtheorem{definition}[theorem]{Definition}
\newtheorem{nota}[theorem]{Notation}

\newtheorem{remark}[theorem]{Remark}

\usepackage{tikz,
}
\usetikzlibrary{cd}

\newcommand{\NN}{\mathbb{N}}

\newcommand{\RR}{\mathbb{R}}
\newcommand{\CC}{\mathbb{C}}
\newcommand{\KK}{\mathbb{K}}
\newcommand{\DD}{\mathbb{D}}

\newcommand{\rank}{\mathrm{rank}}
\DeclareMathOperator{\ran}{ran}

\DeclareMathOperator{\dist}{dist}
\newcommand{\T}{\top}

\newcommand{\Sl}{\left\langle}
\newcommand{\Sr}{\right\rangle}
\newcommand{\Nl}{\left\|}
\newcommand{\Nr}{\right\|}

\renewcommand{\L}{\mathrm{L}}
\renewcommand{\H}{\mathrm{H}}
\newcommand{\C}{\mathrm{C}}

\renewcommand{\d}{\mathrm{d}}
\renewcommand{\epsilon}{\varepsilon}

\renewcommand{\Re}{\mathrm{Re}}
\renewcommand{\Im}{\mathrm{Im}}

\begin{document}

\title{Real and complex stability radii for a class of transport networks}
\author{Anthony Hastir\thanks{University of Namur, Department of Mathematics and naXys, Rue de Bruxelles 61, 5000 Namur, Belgium\linebreak 
  ({anthony.hastir@unamur.be}).}
  \and Birgit Jacob\thanks{University of Wuppertal, School of Mathematics and Natural Sciences, 
Gaußstraße 20, 42119 Wuppertal, Germany ({bjacob@uni-wuppertal.de}, {wagener@uni-wuppertal.de}).}
\and Hannes Wagener\footnotemark[2]}

\maketitle


\section*{Abstract}
We characterize stability and its robustness for a class of boundary controlled, boundary observed hyperbolic partial differential equations. In particular, we show that asymptotic and exponential stability coincide for this class.  Furthermore, we introduce the real and complex stability radii for which we give explicit formulas. In addition, we consider the dual system and its application to controllability and observability. Our main results are illustrated with two examples.

\section{Introduction and main results}
We consider the following hyperbolic partial differential equations (PDEs), 
\begin{align}\tag{\text{$\Sigma_{P_0}$}}\label{eq:WN-P0}
\begin{split}
\frac{\partial {x}}{\partial t}(\xi, t) &= -\frac{\partial}{\partial \xi} \left( \lambda_0(\xi){x}(\xi, t)\right) +P_0(\xi)\lambda_0(\xi){x}(\xi, t),   \quad t\geq 0, \xi\in [0,1], \quad {x}(\xi,0)= {x}_0(\xi),  \quad\xi\in [0,1], \\
\left[\begin{matrix}
0\\I
\end{matrix}\right]u(t)&=-K \lambda_0(0){x}(0, t)-L\lambda_0(1){x}(1,t),  \quad t\geq 0,\\
y(t)&=- K_y \lambda_0(0){x}(0, t)-L_y \lambda_0(1){x}(1,t), \quad t\geq 0,
\end{split}
\end{align}
with initial condition  ${x}_0\in \L^2([0,1];\KK^{n})$, inputs $u(t)\in\KK^m$, $m\leq n$ and outputs $y(t)\in\KK^k$, $t\geq 0$, respectively.
The functions $\lambda_0$ and $P_0$ satisfy $\lambda_0, \lambda_0^{-1}\in \L^\infty([0,1];(0,\infty))$ and $P_0\in \L^\infty([0,1];\KK^{n\times n})$ and  $K,L\in\KK^{n\times n}$, $K_y,L_y\in\KK^{k\times n}$, where $\KK=\RR$ or $\KK=\CC$.
We assume that $K$ is invertible, which is equivalent to well-posedness, see Section \ref{sec:Well-posedness}.

The class \eqref{eq:WN-P0} may model different physical phenomena such as electric lines, the movement of a fluid in a pipe, the deflection of a vibrating string, or more generally conservation laws. An overview of the applications that fall into the class \eqref{eq:WN-P0} may be found in \cite{BasG2016b}. Variations around \eqref{eq:WN-P0} are also available in many different works. In particular, \cite{LuoZ1999b} considers a more general version of \eqref{eq:WN-P0}, in which $\lambda_0$ is replaced by a diagonal matrix with different entries. Many other works considered systems in the form \eqref{eq:WN-P0}. For instance, explicit formula of the Riesz basis and a spectral analysis of \eqref{eq:WN-P0} have been given in \cite{HasA2024a}. An $H^\infty$-optimal control problem has been studied in \cite{HasA2025a3}, while the LQ optimal control problem together with the associated Riccati equations has been treated in \cite{HasA2025a, HasA2025a2}. Lyapunov exponential stability was considered in \cite{BasG2007a,CorJ2007a,DiaA2012a}. In those works, the function $\lambda_0$ has been replaced by a diagonal matrix and the boundary conditions, inputs and outputs are slightly different compared to \eqref{eq:WN-P0}. Backstepping for boundary control and output feedback stabilization for networks of hyperbolic PDEs has been studied in \cite{AurJ2020a,AurJ2021a,VazR2012a}. Many relevant applications in this direction may be found in \cite{VazAurKrst2026}. In the references \cite{deHJ2003a,PriC2008a,PriC2018a} the authors develop boundary feedback control laws for one-dimensional hyperbolic PDEs whose dynamics are close to \eqref{eq:WN-P0}. The class \eqref{eq:WN-P0} has also attracted attention regarding boundary output feedback stabilization, which can be found in e.g.~\cite{TanA2018a}. Feedback stabilization with a numerical point of view has also been treated for a class of hyperbolic PDEs similar to \eqref{eq:WN-P0} in \cite{Gottlich2017}. Some results about stability and robust stabilization for \eqref{eq:WN-P0} may also be found in the literature. For instance, robust stabilization of a class of hyperbolic PDEs has been developed in \cite{Zhang2026}. It is also worth mentioning that under some additional assumption, \eqref{eq:WN-P0} may be embedded into the class of so-called port-Hamiltonian systems, for which a lot of different questions have already been investigated. For instance, stability has been studied in e.g.~\cite{AugB2014a,GerH2024a,VilJ2009a}. Thanks to the port-Hamiltonian formalism, equivalent conditions for exponential stability of \eqref{eq:WN-P0} have also been discovered in \cite{HasA2024a} thanks to a spectral analysis. Controllability has also been studied for hyperbolic PDEs similar to \eqref{eq:WN-P0}. In particular, the approximate and exact controllability has already been investigated for a system similar to \eqref{eq:WN-P0} in \cite{ChiY2024a}, where $\lambda_0$ is replaced by a diagonal matrix with possibly different entries and where $P_0$ is also assumed to be diagonal. In the case of commensurable wave speeds, the authors of \cite{ChiY2024a} show that exact and approximate controllability are equivalent and can be characterized by some matrix tests.

In this paper we focus on the stability and on the robustness of stability of \eqref{eq:WN-P0} under output feedback. Moreover we discuss the real stability radius for this class. For more details about the stability radius, we refer the reader to~\cite{HinD1990a} in which linear finite-dimensional systems are considered and \cite{PriA1990a} for an extension to infinite dimension. More precisely, we show that exponential and asymptotic stability are equivalent for \eqref{eq:WN-P0} and may be characterized by a simple matrix condition. 

Furthermore, we consider the dual system of \eqref{eq:WN-P0} and its application to controllability and observability. In particular, we take advantage of the discrete-time approach which we use for the stability of robustness to obtain concrete matrix tests for the controllability with matrices which are based on $K,L,K_y,L_y$ and $P_0$. Thanks to a duality result we do characterize observability via similar matrix conditions.

Before going further, we spend a few lines in giving a meaning to $x$ and $y$ in \eqref{eq:WN-P0}. This brings us to the concept of well-posedness, by which we mean the following: \eqref{eq:WN-P0} is said to be well-posed if for every $u\in\L^2_{\mathrm{loc}}([0,\infty);\KK^m)$ and for every $x_0\in \L^2([0,1],\KK^n)$, there exists a unique ${x}\in \C([0,\infty);\L^2([0,1],\KK^n))$ which is continuously dependent on $x_0$ and $u$, such that the output $y$ satisfies $y\in\L^2_{\mathrm{loc}}([0,\infty);\KK^k)$. For further details we refer to Section \ref{sec:Well-posedness}. 
In particular, the system \eqref{eq:WN-P0} is well-posed if and only if the matrix $K$ is invertible. Under this condition, we may define the following matrices
\begin{align}\label{eq:Mat}
\begin{split}
A &:= -K^{-1}LP(1)\in\KK^{n\times n},\hspace{0.5cm} B := -K^{-1}\left[\begin{smallmatrix}0\\ I\end{smallmatrix}\right]\in\KK^{n\times m},\\
C &:= (K_y K^{-1}L - L_y)P(1)\in\KK^{k\times n},\hspace{0.5cm} D := K_y K^{-1}\left[\begin{smallmatrix}0\\ I\end{smallmatrix}\right]\in\KK^{k\times m},    
\end{split}
\end{align}
where $P(\xi)$ is given by the solution of the initial value problem \begin{align*}
P'(\xi)&=P_0(\xi)P(\xi), \quad \xi\in(0,1], \quad P(0)=I.
\end{align*}
Especially, if $P_0\equiv0$, we have $P(1)=I$ and if $P_0$ is constant, we have $P(1)=\mathrm{e}^{P_0}$.

In Section \ref{sec:Well-posedness} we show, that the matrices introduced in \eqref{eq:Mat} are crucial and allow us to rewrite \eqref{eq:WN-P0} equivalently in the much simpler form
\begin{align*}
\begin{split}
\frac{\partial \tilde{x}}{\partial t}(\xi, t) &= -\frac{\partial}{\partial \xi} \left( \lambda_0(\xi)\tilde{x}(\xi, t)\right),   \quad t\geq 0, \xi\in [0,1], \quad\tilde{x}(\xi,0)= \tilde{x}_0(\xi),  \quad\xi\in [0,1], \\
B u(t)&=\lambda_0(0)\tilde{x}(0, t)-A\lambda_0(1)\tilde{x}(1,t),  \quad t\geq 0,\\
y(t)&=C\lambda_0(1)\tilde{x}(1, t)+D u(t), \quad t\geq 0,
\end{split}
\end{align*}
where $x$ and $\tilde{x}$ satisfy the relation  $\tilde{x}(\xi,t) = P(\xi)^{-1}x(\xi,t)$, $\xi\in [0,1]$, $t\geq 0$.

In this paper we derive tractable conditions for the stability and the robustness of stability of the infinite‑\allowbreak dimensional system \eqref{eq:WN-P0} that depend only on the associated constant matrices $A$, $B$, $C$ and $D$. The cornerstone of our approach is an equivalent reformulation of \eqref{eq:WN-P0} as an infinite‑dimensional discrete‑time system whose dynamics are governed by multiplication operators with constant matrices. This representation allows us to prove that the stability (resp.~robust stability) of the original system can be inferred from a finite‑dimensional system constructed solely from the matrices $A$, $B$, $C$ and $D$. Consequently, one can analyse the infinite‑dimensional problem by means of standard finite‑dimensional tools.

The remaining of this section is dedicated to the exposition of our main results, whose proofs can be found in Section~\ref{sec:FinAppr}.
We now assume that the system \eqref{eq:WN-P0} is (asymptotically) stable and aim to characterize the smallest perturbation (measured in a prescribed norm) that destabilizes \eqref{eq:WN-P0}. The perturbations under consideration arise from static output feedback of the form  
$u(t) = \Delta y(t)$ with $\Delta\in\KK^{m\times k}$. 
When this feedback law is applied to \eqref{eq:WN-P0}, we have that
\begin{align}\tag{\text{$\Sigma_{P_0}^\Delta$}}\label{eq:WN-P0-pert}
\begin{split}
\frac{\partial {x}}{\partial t}(\xi, t) &= -\frac{\partial}{\partial \xi} \left( \lambda_0(\xi){x}(\xi, t)\right) +P_0(\xi)\lambda_0(\xi){x}(\xi, t),   \quad t\geq 0, \xi\in [0,1], \quad {x}(\xi,0)= {x}_0(\xi),  \quad\xi\in [0,1], \\
0&= -\left(K-\left[\begin{smallmatrix}
0\\I
\end{smallmatrix}\right]\Delta K_y\right)\lambda_0(0){x}(0, t) - \left(L-\left[\begin{smallmatrix}
0\\I
\end{smallmatrix}\right]\Delta L_y\right) \lambda_0(1){x}(1,t), \, t\geq 0.
\end{split}
\end{align}

 According to Corollary~\ref{lem:Well-Posedness_Cl}, the perturbed system \eqref{eq:WN-P0-pert} is well-posed, that is, the corresponding operator generates a strongly continuous semigroup, if and only if 
$I-\Delta D$ is invertible, where $D$ is defined in \eqref{eq:Mat}.
Consequently, we henceforth assume that $I-\Delta D$ is invertible. Our notion of stability is formalized in the following definition.

\begin{definition}
    Let $\Delta\in\KK^{m\times k}$ and assume that $I-\Delta D$ is invertible. The system \eqref{eq:WN-P0-pert} is asymptotically stable, if for all $x_0\in \L^2([0,1];\KK^n)$ we have \begin{align*}
        \lim_{t\to\infty} x(\cdot,t) = 0
    \end{align*} in $\L^2([0,1];\KK^n)$,  where $x$ denotes the solution of \eqref{eq:WN-P0-pert}. If, in addition, there exist $\omega>0$ and $M\geq 1$ such that for all $x_0\in \L^2([0,1];\KK^n)$ and all $t>0$ we have that \begin{align*}
         \Nl x(\cdot,t) \Nr_{\L^2([0,1]; \KK^n)} \leq M \mathrm{e}^{-\omega t}  \Nl x_0 \Nr_{\L^2([0,1]; \KK^n)},
     \end{align*}
     then \eqref{eq:WN-P0-pert} is said to be exponentially stable.
\end{definition}

One of our main results is that asymptotic and exponential stability are equivalent for \eqref{eq:WN-P0-pert}, and can be characterized by a simple matrix condition.

\begin{proposition}\label{prop:MainresStab}
   Let $A,B,C$ and $D$ be defined as in \eqref{eq:Mat}, $\Delta\in\KK^{m\times k}$ and assume that $I-\Delta D$ is invertible. \eqref{eq:WN-P0-pert} is asymptotically stable if and only if it is exponentially stable, which is equivalent to\footnote{We denote with $\sigma(J)$ the set of all eigenvalues of the matrix $J$.}\linebreak $\sigma\left(A+B(I-\Delta D)^{-1}\Delta C\right)\subset \DD:=\{z\in\CC: |z|<1\}$.
\end{proposition}

We say that \eqref{eq:WN-P0} is \emph{asymptotically stable} or \emph{exponentially stable}, if the state $x(\cdot,t)$ of the homogeneous system (i.e.~\eqref{eq:WN-P0} with $u(t)\equiv0$) tends to zero or tends exponentially fast to zero for $t\to\infty$, respectively.
The homogeneous system is the perturbed system \eqref{eq:WN-P0-pert} when $\Delta=0$. Consequently, stability can be characterized by using Proposition \ref{prop:MainresStab}. Therefore, \eqref{eq:WN-P0} is asymptotically stable if and only if \eqref{eq:WN-P0} is exponentially stable, which is also equivalent to $\sigma(A)\subset \DD$. Since asymptotic and exponential stability are equivalent for \eqref{eq:WN-P0} and \eqref{eq:WN-P0-pert}, we write shortly that \eqref{eq:WN-P0} or \eqref{eq:WN-P0-pert} are \emph{stable}, when they are asymptotically or exponentially stable.

By robustness of stability we mean the following question: how \emph{large} can the perturbation $\Delta$ be while the perturbed system remains stable? To address this, we introduce the stability radius, distinguishing between real‑valued and complex‑valued perturbations.
\begin{definition}\label{def:stab_radius}
    For a stable system \eqref{eq:WN-P0} the stability radius is defined by
    \begin{align*}
        r_\KK\eqref{eq:WN-P0}&:= \inf \left\lbrace \Nl \Delta \Nr_{\KK^{m\times k}} : \Delta \in \KK^{m\times k}, I-\Delta D \text{ is not invertible or }\eqref{eq:WN-P0-pert} \text{ is not stable} \right\rbrace.
    \end{align*}
\end{definition}
The stability radius can also be expressed via a supremum in the following way \begin{align*}
        r_\KK\eqref{eq:WN-P0}&= \sup \left\lbrace r>0 :   \text{For all } \Delta \in \KK^{m\times k} \text{ with }\Nl \Delta \Nr_{\KK^{m\times k}}<r \text{ hold } I-\Delta D \text{ is invertible and }\eqref{eq:WN-P0-pert} \text{ is stable.} \right\rbrace.
    \end{align*}

Our main result concerning the robustness of stability is the following:

\begin{theorem}\label{thm:MainresRob}
  Let the matrices $A,B,C$ and $D$ be defined as in \eqref{eq:Mat} and let \eqref{eq:WN-P0} be stable (equiv.~$\sigma(A)\subset \DD$). We consider\footnote{The notation $\rho(A)$ stands for the resolvent set of $A$, i.e., $\rho(A) := \CC\setminus\sigma(A)$.} $G(z)=C(z-A)^{-1} B+D\in\CC^{k\times m}$, $z\in \rho(A)$. Then the complex stability radius is given by
     \begin{align*}
         r_\CC\eqref{eq:WN-P0}= \left[{\max_{z\in\partial \DD} \Nl G(z) \Nr_{\CC^{k\times m}} }\right]^{-1}.
     \end{align*}
 
For the real stability radius we  assume that all norms are induced by the Euclidean vector-norm.
  In the case where $m=1$ or $k=1$, the real stability radius takes the form
\begin{align*}
   r_\RR\eqref{eq:WN-P0}= \left[ \max \left\lbrace\Nl D\Nr_{\KK^{k\times m}} , \max_{z\in\partial\DD} \dist\left( \Re (G(z)), \RR \Im (G(z)) \right) \right\rbrace \right]^{-1}
\end{align*}
    and when $m,k>1$ we obtain
          \begin{align*}
         r_\RR\eqref{eq:WN-P0}=\left[{\max\left\lbrace \Nl D \Nr_{\KK^{k\times m}} , \max_{z\in\partial \DD}\inf_{\gamma\in (0,1]} \sigma_2\left(\left[\begin{smallmatrix}
        \Re(G(z)) & -\gamma \Im(G(z)) \\
        \frac{1}{\gamma}\Im(G(z)) & \Re(G(z))
    \end{smallmatrix}\right]\right) \right\rbrace}\right]^{-1}
     \end{align*}
    where $\sigma_2(J)$ stands for the second largest singular value of a matrix $J$. In all cases the convention $\frac{1}{\infty}=0$ holds.         
\end{theorem}

Note that the relation $r_\CC\eqref{eq:WN-P0}\leq r_\RR\eqref{eq:WN-P0}$ always holds. 
However, Example \ref{sec:Ex-nice} shows that equality does not necessarily hold in general.
 Therefore it is worth to distinguish between the real and complex stability radius. Furthermore, it can be shown that the infimum in Definition \ref{def:stab_radius} is actually a minimum. Moreover, in the case where $r_\KK\eqref{eq:WN-P0}$ is finite we can construct a perturbation with minimal norm which destabilizes \eqref{eq:WN-P0-pert} based on the matrices defined in \eqref{eq:Mat}. This perturbation often can be constructed with rank 1. Further details are highlighted in Remark \ref{rem:MinNormPertINFITY}.

\subsection*{Outline of this paper}
This paper is organized as follows. In Section \ref{sec:Well-posedness}, we review known results on well-posedness and transform the system so as to eliminate the matrix function $P_0$ in \eqref{eq:WN-P0}. Our approach to analysing \eqref{eq:WN-P0} is to reformulate the system as an equivalent infinite-dimensional discrete-time system, as carried out in Section \ref{sec:DiscreteAppr}. One of the main advantages of this approach is that the operators of the discrete-time system are given by multiplication operators with constant matrices. Our ansatz is to consider a finite-dimensional discrete-time defined with those matrices. The results concerning the stability and its robustness for the original system are then derived from the same types of results for the finite-dimensional system. Therefore we give a short introduction into linear finite-dimensional discrete-time systems in Section \ref{sec:IntrFinDimSyst} and apply this to our network of waves in Section \ref{sec:FinAppr}, where we also prove the aforementioned main results. In addition we use the discrete-time expression to calculate the dual system and use the finite-dimensional approach to state some matrix conditions for controllability and observability, see Section \ref{sec:DualAndContrObs}. Several examples are presented in Section \ref{sec:Ex}, while Section \ref{sec:Persp} is dedicated to the perspectives.

\section{Well-posedness and simplification}\label{sec:Well-posedness}
In this section we summarize known results on well-posedness of \eqref{eq:WN-P0}. First we eliminate the matrix function $P_0$ in \eqref{eq:WN-P0} via variables following the construction in~\cite[Lem.~2.1 and 2.2]{HasA2025a}.
\begin{lemma}\label{lem:ZustTrafo} 
The initial value problems 
\begin{align}
\begin{split}
Q'(\xi)&=-Q(\xi)P_0(\xi), \quad\, \xi\in(0,1],  Q(0)=I, \\
P'(\xi)&=P_0(\xi)P(\xi), \qquad \xi\in(0,1],  P(0)=I 
\label{eq:ODE_Q_P}
\end{split}
\end{align}
have unique Carathéodory-solutions $Q$ and $P$, i.e., $Q$ and $P$ are absolutely continuous on $[0,1]$ and fulfill the initial value problem almost everywhere. For every $\xi \in [0,1]$ the matrices $Q(\xi)$ and $P(\xi)$ are invertible with \begin{align*}
Q(\xi) = P(\xi)^{-1}.
\end{align*}
\end{lemma}

With the solutions of \eqref{eq:ODE_Q_P} we can transform \eqref{eq:WN-P0} and eliminate  $P_0$.
\begin{proposition}[Simplification]\label{prop:Simpl}
The system \eqref{eq:WN-P0} is equivalent to the boundary controlled PDEs
\begin{align}\label{eq:WN-Simp}\begin{split}
\frac{\partial \tilde{x}}{\partial t}(\xi,t) &= -\frac{\partial}{\partial \xi}(\lambda_0(\xi)\tilde{x}(\xi,t)), \quad t\geq 0,\xi\in [0,1],\quad \tilde{x}(\xi,0) = \tilde{x}_0(\xi),  \quad \xi\in [0,1], \\
\left[\begin{smallmatrix}0\\ I\end{smallmatrix}\right]u(t) &= - K\lambda_0(0)\tilde{x}(0,t) -  L P(1)\lambda_0(1)\tilde{x}(1,t), \quad  t\geq 0,\\ 
y(t) &= - K_y\lambda_0(0) \tilde{x}(0,t) -  L_y P(1)\lambda_0(1) \tilde{x}(1,t),  \quad t\geq 0,
\end{split}
\end{align}
with initial condition $\tilde{x}_0\in \L^2([0,1];\KK^{n})$, control $u(t)\in\KK^m$ and observation $y(t)\in\KK^k$, where $x$ and $\tilde{x}$ satisfy the relation 
\begin{align*}
\tilde{x}(\xi,t)=Q(\xi)x(\xi,t) \quad \text{ and } \quad {x}(\xi,t)=P(\xi)\tilde{x}(\xi,t),
\end{align*} with $Q$ and $P$ given by  \eqref{eq:ODE_Q_P}.  
\end{proposition}

A first important question related to boundary controlled, boundary observed systems is well-posedness, which we explain in the following definition. For a general introduction, we refer the reader to e.g.~\cite[Chap.~13]{JacB2012b}.
 \begin{definition}\label{defi:well-posed}
     The system \eqref{eq:WN-Simp} is said to be a \emph{well-posed} boundary controlled system, if for all $t>0$ there exists $m_{t}>0$ such that for all initial values $\tilde{x}_0\in \left\{\tilde{x}\in\L^2([0,1];\KK^n):\lambda_0\tilde{x}\in \H^1([0,1],\KK^n)\right\}=:Z$ and input functions $u\in \C^2([0,t];\KK^m)$  with $\left[\begin{smallmatrix}
             0\\ u(0)
         \end{smallmatrix}\right]=-K (\lambda_0(0){\tilde{x}_0})(0)-LP(1)(\lambda_0(0){\tilde{x}_0})(1),$ we can define unique solutions $\tilde{x}\in \C^1([0,t];\L^2([0,1];\KK^n))$ with $\tilde{x}(t)\in Z$ and $y\in\C([0,t],\KK^k)$ of \eqref{eq:WN-Simp} and we have         \begin{align}\label{eq:Well-posedness-Ineq}
\Nl \tilde{x}(t) \Nr_{\L^2([0,1]; \KK^n)}^2 + \int_0^{t} \Nl y(s) \Nr_{\KK^k}^2 \d{}s \leq m_{t} \left( \Nl \tilde{x}_0 \Nr_{\L^2([0,1]; \KK^n)}^2 +\int_0^{t} \Nl u(s) \Nr_{\KK^m}^2 \d{}s \right).
 \end{align}
 \end{definition}
\begin{remark}\begin{itemize}
\item It would be sufficient to ask that \eqref{eq:Well-posedness-Ineq} holds for one $t>0$ to infer that it is true for all $t>0$, see \cite[Thm.~13.1.7]{JacB2012b}.
\item Setting $u\equiv0$ in \eqref{eq:Well-posedness-Ineq} gives $\Nl \tilde{x}(t)\Nr_{\L^2([0,1]; \KK^n)}\leq m_t \Nl \tilde{x}_0\Nr_{\L^2([0,1]; \KK^n)}$, which is equivalent to well-posedness of the homogeneous system, that is, the operator \begin{align}\label{eq:HomoOper}
                \mathcal{A}f = -\frac{\d}{\d\xi}(\lambda_0 f),\quad
                D(\mathcal{A}) = \left\{f\in Z, 
                K(\lambda_0 f)(0) = -LP(1)(\lambda_0 f)(1) \right\}   
            \end{align} generates a strongly continuous semigroup.
            \item Choosing $u\equiv0$ and $\tilde{x}_0=0$ in \eqref{eq:Well-posedness-Ineq} gives $\int_0^{t} \Nl y(s) \Nr_{\KK^k}^2 \d{}s \leq m_{t}  \Nl \tilde{x}_0 \Nr_{\L^2([0,1]; \KK^n)}^2 $ and $\Nl \tilde{x}(t) \Nr_{\L^2([0,1]; \KK^n)}^2  \leq m_{t} \int_0^{t} \Nl u(s) \Nr_{\KK^m}^2 \d{}s$, respectively. Those inequalities reflect the admissibility of the output and input operators, respectively. For more details on admissibility of unbounded operators, we refer to \cite{TucM2009b}.
    \item The inequality \eqref{eq:Well-posedness-Ineq} shows that the mapping from $\left[\begin{smallmatrix}\tilde{x}_0\\u\end{smallmatrix}\right]$ to $\left[\begin{smallmatrix}\tilde{x}\\y\end{smallmatrix}\right]$ is bounded and it is densely defined in $\L^2([0,1];\KK^n)\times \L^2([0,t];\KK^m)$. Therefore it can be extended continuously to a mapping from $\L^2([0,1];\KK^n)\times \L^2([0,t];\KK^m)$. This extension leads to the concept of \emph{mild solution}. 
\end{itemize}
\end{remark}
Well-posedness for \eqref{eq:WN-P0} is addressed in the following proposition, see \cite[Thm.~3.3]{ZwaH2010a}.

\begin{proposition}\label{satz:well-posedness}
The system \eqref{eq:WN-Simp} (equivalently \eqref{eq:WN-P0}) is well-posed if and only if $K$ is invertible. Especially, if $K$ is not invertible, the operator describing the homogeneous dynamics, that is, \eqref{eq:HomoOper},
does not generate a strongly continuous semigroup.
\end{proposition}

\begin{corollary}\label{lem:Well-Posedness_Cl}
Let $\Delta\in\KK^{m\times k}$. \eqref{eq:WN-P0-pert}
is well-posed if and only if $I - \Delta D$ is invertible.
\end{corollary}

\begin{proof}
    Thanks to Proposition \ref{satz:well-posedness}, \eqref{eq:WN-P0-pert} is well-posed if and only if $\left[\begin{smallmatrix}0\\ I\end{smallmatrix}\right]\Delta K_y-K$ is invertible, which is also equivalent to the invertibility of $I-\left[\begin{smallmatrix}0\\ I\end{smallmatrix}\right]\Delta K_yK^{-1}$. Moreover, it can be shown that $I-\left[\begin{smallmatrix}0\\ I\end{smallmatrix}\right]\Delta K_yK^{-1}$ is invertible if and only if $I-\Delta K_yK^{-1}\left[\begin{smallmatrix}0\\ I\end{smallmatrix}\right]$ is invertible. The proof concludes by noting that $I-\Delta K_yK^{-1}\left[\begin{smallmatrix}0\\ I\end{smallmatrix}\right] = I-\Delta D$.
\end{proof}

\begin{remark}
    \begin{itemize}
        \item Proposition \ref{satz:well-posedness} also implies that the homogeneous closed-loop system composed of \eqref{eq:WN-Simp} with $u(t) = \Delta y(t)$ is well-posed if and only if $I-\Delta D$ is invertible, which is equivalent to the semigroup generation of the operator $(\mathcal{A}_\mathrm{cl},D(\mathcal{A}_\mathrm{cl}))$ defined by
        \begin{subequations}
            \begin{align*}
                \mathcal{A}_{\mathrm{cl}}f &= -\frac{\d}{\d\xi}(\lambda_0 f),\quad
                D(\mathcal{A}_\mathrm{cl}) = \left\{f\in Z,\left(K-\left[\begin{smallmatrix}0\\ I\end{smallmatrix}\right]\Delta K_y\right)(\lambda_0 f)(0) = -  \left(L-\left[\begin{smallmatrix}0\\ I\end{smallmatrix}\right]\Delta L_y\right) P(1)(\lambda_0 f)(1) \right\}.
            \end{align*}
        \end{subequations}
        \item The invertibility of $I-\Delta D$ has already been shown a necessary condition for the well-posedness of the closed-loop system in \cite[Prop.~4.6]{WeiG1994a}. In Corollary~\ref{lem:Well-Posedness_Cl}, we actually show that this condition is also sufficient for our class of systems. 
    \end{itemize}
\end{remark}
We conclude this section by expressing \eqref{eq:WN-Simp} in an equivalent form using the matrices $A$, $B$, $C$  and $D$ defined in \eqref{eq:Mat}. The invertibility of $K$ allows us to express \eqref{eq:WN-Simp} as
\begin{align}\tag{\text{$\Sigma_0$}}\label{eq:WN-0}
\begin{split}
\frac{\partial \tilde{x}}{\partial t}(\xi, t) &= -\frac{\partial}{\partial \xi} \left( \lambda_0(\xi)\tilde{x}(\xi, t)\right),   \quad t\geq 0, \xi\in [0,1],\quad \tilde{x}(\xi,0)= \tilde{x}_0(\xi),  \quad\xi\in [0,1], \\
B u(t)&=\lambda_0(0)\tilde{x}(0, t)-A\lambda_0(1)\tilde{x}(1,t),  \quad t\geq 0,\\
y(t)&=C\lambda_0(1)\tilde{x}(1, t)+D u(t), \quad t\geq 0.
\end{split}
\end{align}

For the robustness analysis, we consider the closed-loop system obtained by interconnecting\eqref{eq:WN-0} with $u(t)=\Delta y(t)$, which, under the invertibility of $I-\Delta D$, reads
\begin{align}\tag{\text{$\Sigma_0^\Delta$}}\label{eq-WN-0-pert}
\begin{split}
\frac{\partial \tilde{x}}{\partial t}(\xi, t) &= -\frac{\partial}{\partial \xi} \left( \lambda_0(\xi)\tilde{x}(\xi, t)\right),   \quad t\geq 0, \xi\in [0,1],\quad \tilde{x}(\xi,0)= \tilde{x}_0(\xi),  \quad\xi\in [0,1], \\
0&=\lambda_0(0)\tilde{x}(0, t)-\left(A +B (I-\Delta D)^{-1}\Delta C \right)\lambda_0(1)\tilde{x}(1,t),  \quad t\geq 0.
\end{split}  
\end{align}

\section{Equivalent discrete-time formulation}\label{sec:DiscreteAppr}
In this section we show that the network of waves \eqref{eq:WN-0} can be expressed equivalently as a discrete-time system. After that, we define the concepts of stability on the discrete-time level and show that they are equivalent to the same concepts in the continuous-time level. We start by defining the functions $k:[0,1]\to [0,1]$ and $p:[0,1]\to [0,\infty)$ by
\begin{align}\label{lem:Fctkp}
k(\xi) := 1-\frac{p(\xi)}{p(1)}, \quad p(\xi) := \int_0^\xi \frac{1}{\lambda_0(\zeta)}\d\zeta, \quad \xi\in[0,1],
\end{align}
respectively. It follows from basic calculus that $k$ is a homeomorphism.

The following proposition, based on \cite[Lem.~3.1]{HasA2025a}, provides an equivalent discrete-time representation of the system \eqref{eq:WN-0}. The proof is presented in a slightly different form than in \cite[Lem.~3.1]{HasA2025a} and is included here for completeness.

\begin{nota}
For a matrix $J\in\KK^{j\times l}$ we denote by $M_J$ the associated multiplication operator\linebreak $M_J: \L^2([0,1]; \KK^l) \to \L^2([0,1]; \KK^j), \quad x\mapsto Jx$.   
\end{nota}

\begin{proposition}\label{satz:Diskretisierung}
We consider the system \eqref{eq:WN-0} with matrices $A,B,C$ and $D$ be defined as in \eqref{eq:Mat}. Further let the functions $k$ and $p$ be defined as in \eqref{lem:Fctkp}.
Then \eqref{eq:WN-0} is equivalent to the infinite-dimensional discrete-time system
\begin{align}\tag{\text{$\Sigma_d$}}\label{eq:WN-disc}\begin{split}
x_d(\tau+1) &= M_{A} x_d(\tau) + M_{B} u_d(\tau),\quad \tau\in\NN_0,\quad x_d(0) = x_{d,0},\\
y_d(\tau) &= M_{C} x_d(\tau) + M_{D} u_d(\tau), \quad \tau\in\NN_0,
\end{split}
\end{align}
with state space $\L^2([0,1];\KK^n)$, initial condition $x_{d,0}=\lambda_0(k^{-1}(\cdot))\tilde{x}_0(k^{-1}(\cdot))$, inputs $u_d(\tau)\in \L^2([0,1];\KK^m)$ and outputs $y_d(\tau) \in  \L^2([0,1];\KK^k)$, $\tau\in\NN_0$. 

The inputs and the outputs of \eqref{eq:WN-0} and \eqref{eq:WN-disc} are related as follows:
\begin{align*}
u((\tau+\zeta)p(1))=u_d(\tau)(\zeta), \quad \tau\in\NN_0,\zeta\in [0,1] ,\\
y((\tau+\zeta)p(1))=y_d(\tau)(\zeta), \quad \tau\in\NN_0,\zeta\in [0,1].
\end{align*}
If we define $f$ on $[0,\infty)$ by $f(\tau+\zeta):=x_d(\tau)(\zeta), \tau\in\NN_0,\zeta\in [0,1]$, we obtain \begin{align*}
\tilde{x}(\xi,t)=\lambda_0(\xi)^{-1}f\left(k(\xi)+ p(1)^{-1}t \right), \quad t\geq0, \xi\in [0,1].
\end{align*}

We notice that for every $t>0$ there exist a $\tau\in\NN_0$ and a $\zeta\in[0,1]$ such that $t=\tau+\zeta$. Thus, we have a unique relation between both systems.

\end{proposition}

\begin{proof}
We start with \eqref{eq:WN-disc}. Let the solution $x_d$, the input $u_d$ and the output $y_d$ be given. For $0\leq t= (\tau +\zeta)p(1)$ with $\tau\in\NN_0$ and $\zeta\in[0,1]$ we define \begin{align*}
f\left( p(1)^{-1}t \right)=f(\tau+\zeta)=x_d(\tau)(\zeta),\\
u(t)=u((\tau+\zeta)p(1))=u_d(\tau)(\zeta) ,\\
y(t)= y((\tau+\zeta)p(1))=y_d(\tau)(\zeta).
\end{align*}
With this setting \eqref{eq:WN-disc} is equivalent to \begin{align*}
f\left( p(1)^{-1}t +1\right) &= {A} f\left( p(1)^{-1}t \right) + {B} u(t),\quad t\geq 0,\\
f(\zeta) &= x_{d,0}(\zeta),\quad \zeta\in [0,1],\\
y(t) &= {C} f\left( p(1)^{-1}t \right) + {D} u(t), \quad t\geq 0.
\end{align*}
If we define $
\tilde{x}(\xi,t):=\lambda_0(\xi)^{-1}f\left(k(\xi)+ p(1)^{-1}t \right)$, $ t\geq0, \xi\in [0,1],$
we get that 
\begin{align*}
\tilde{x}(\xi,0) &= \tilde{x}_0(\xi),\quad \xi\in [0,1],\\
B u(t) &= \lambda_0(0)\tilde{x}(0,t) -  A\lambda_0(1)\tilde{x}(1,t), \quad  t\geq 0,\\
y(t) &= C\lambda_0(1) \tilde{x}(1,t) + Du(t),  \quad t\geq 0.
\end{align*}
Since the solution of a well-posed boundary control system is unique, noting that $\tilde{x}$ fulfills the differential equation in \eqref{eq:WN-0} concludes the proof.
\end{proof}

    Note that the solution of \eqref{eq:WN-disc} at time $\tau\in\NN_0$ is given by \begin{align*}
x_d(\tau)=& M_A^\tau x_{d,0}+\sum_{\sigma=0}^{\tau-1} M_A^{\tau-1-\sigma}M_B u_d(\sigma), \quad
y_d(\tau)= M_C M_A^\tau x_{d,0}+\sum_{\sigma=0}^{\tau-1} M_CM_A^{\tau-1-\sigma}M_B u_d(\sigma)+M_D u_d(\tau),
\end{align*} where $x_{d,0}\in\L^2([0,1];\KK^n)$ is the initial value and $u_d(0),u_d(1),\ldots,$ $u_d(\tau-1)\in \L^2([0,1];\KK^m)$ are the inputs. 
To transfer the stability properties from \eqref{eq:WN-0} to \eqref{eq:WN-disc}, we require an explicit relation between the states of the two systems \eqref{eq:WN-0} and \eqref{eq:WN-disc}, which is summarized in the next remark.
\begin{remark}\label{rem:stateRel}
   For $\tau\in\NN_0$ and $\zeta\in[0,1]$ we have \begin{align*}
    x_d(\tau)(\zeta)= f(\tau+\zeta) = \lambda_0(k^{-1}(\zeta)) \tilde{x}\left(k^{-1}(\zeta), \tau p(1)\right).
\end{align*}
 Moreover, if we pick $t>0$ with $ p(1)^{-1}t \in \NN$ and $\xi\in[0,1]$, then it holds that
 \begin{align*}
\tilde{x}(\xi,t)=\lambda_0(\xi)^{-1} f\left( k(\xi)+  p(1)^{-1}t \right)=\lambda_0(\xi)^{-1} x_d\left( p(1)^{-1}t \right)(k(\xi)).
\end{align*} 
For arbitrary $\xi\in [0,1]$ and $t>0$, which we split as $ p(1)^{-1}t =\tau+\delta$ with $\tau\in\NN_0$ and $\delta\in[0,1)$, we have
\begin{align*}
&    \tilde{x}(\xi,t)= \tilde{x}(\xi, p(1)(\tau+\delta) )=\lambda_0(\xi)^{-1} f\left( k(\xi)+ \tau+\delta\right)
    =\begin{cases}
     \lambda_0(\xi)^{-1} x_d(\tau)(k(\xi)+\delta),& \text{ if } \,\,k(\xi)+\delta\leq1, \\
      \lambda_0(\xi)^{-1} x_d(\tau+1)(k(\xi)+\delta-1),& \text{ if } \,\,k(\xi)+\delta>1.
    \end{cases}
\end{align*}
\end{remark}

Our objective is to use the link between the continuous- and discrete-time systems \eqref{eq:WN-0} and \eqref{eq:WN-disc} to study the stability of \eqref{eq:WN-P0-pert}. 
Since the invertibility of $I-\Delta D$ is an equivalent condition for the well-posedness of \eqref{eq:WN-P0-pert}, we shall make the assumption that $I-\Delta D$ is invertible throughout this section.

\begin{lemma}\label{lem:ClosedDiscTimeInfDim}
     Let $\Delta\in\KK^{m\times k}$ with $I-\Delta D$ invertible. Then, the closed-loop system composed of \eqref{eq:WN-disc} and $u_d(\tau)=M_\Delta y_d(\tau)$ is well-posed, where by \emph{well-posed} we mean that the closed-loop system admits a unique solution.
\end{lemma}

\begin{proof}
         It suffices to note that $I-\Delta D$ is invertible if and only if $I-M_\Delta M_D$ is invertible.
         \end{proof}

\begin{remark}
    The system composed of \eqref{eq:WN-disc} and $u_d(\tau)=M_\Delta y_d(\tau)$ could be understood as the structured perturbation of \eqref{eq:WN-disc}. It reads as \begin{align}\tag{\text{$\Sigma_d^\Delta$}}\label{eq:WN-disc-pert}
        \begin{split}
            x_d(\tau+1)&=\left(M_A+M_B(I-M_\Delta M_D)^{-1} M_\Delta M_C\right) x_d(\tau), \quad \tau\in\NN_0,\quad   x_d(0)=x_{d,0}.
        \end{split}
    \end{align}
\end{remark}

We define stability for the discrete-time setting as follows.
\begin{definition}\label{def:StabDiscr}
    Let $\Delta\in\KK^{m\times k}$ with $I-\Delta D$ is invertible. The system \eqref{eq:WN-disc-pert}, whose state is denoted by $x_d$, is said to be asymptotically stable, if for all $x_{d,0}\in \L^2([0,1],\KK^n)$ we have  \begin{align*}
        \lim_{\tau\to\infty,\tau\in\NN} x_d(\tau) = 0.
    \end{align*} 

      If there exist $\omega>0$ and $M\geq 1$ such that for all $x_{d,0}\in \L^2([0,1],\KK^n)$ and all $\tau\in\NN_0$ the state of \eqref{eq:WN-disc-pert}, denoted by $x_d$, satisfies \begin{align*}
         \Nl x_d(\tau) \Nr_{\L^2([0,1]; \KK^n)} \leq M \mathrm{e}^{-\omega \tau}  \Nl x_{d,0} \Nr_{\L^2([0,1]; \KK^n)},
     \end{align*}
     then the system \eqref{eq:WN-disc-pert} is exponentially stable.
\end{definition}

\begin{remark}
  In the discrete-time literature, it is often more common to refer to strong and power stability rather than to asymptotic and exponential stability, respectively.   According to \cite[Chap.~4]{CurR2020b}, \eqref{eq:WN-disc} is strongly stable if $(M_A)^\tau x\to 0$ when $\tau\to\infty$ for all $x\in \L^2([0,1];\KK^n)$, which is equivalent to\footnote{By $x_d$ in this context, we mean the state of \eqref{eq:WN-disc} with $u\equiv 0$.} $x_d(\tau)\to 0$ when $\tau\to\infty$ for all $x_{d,0}\in \L^2([0,1];\KK^n)$. In particular, this means that asymptotic and strong stability coincide in this context. Moreover, \eqref{eq:WN-disc} is said to be power stable if there exist $M\geq 1$ and $\gamma\in(0,1)$ such that $\Vert (M_A)^\tau\Vert_{\mathcal{L}(\L^2([0,1];\KK^n))}\leq M\gamma^\tau$, for all $\tau\in\mathbb{N}$, see e.g.\cite[Chap.~4]{CurR2020b}. Making the identification $\gamma = e^{-\omega}$ in Definition \ref{def:StabDiscr} shows that power and exponential stability are equivalent concepts.
\end{remark}

In the next lemma we show that the stability concepts for the continuous-time system \eqref{eq:WN-0} and the discrete-time systems \eqref{eq:WN-disc} coincide.
\begin{lemma}\label{lem:Stab-Wn-Disc}
    Let $\Delta\in\KK^{m \times k}$ and assume that $I-\Delta D$ is invertible.
    The system \eqref{eq-WN-0-pert} is asymptotically or exponentially stable, if and only if \eqref{eq:WN-disc-pert} is asymptotically or exponentially stable, respectively.
\end{lemma}

\begin{proof} 
Let $\Delta\in\KK^{m \times k}$ with $I-\Delta D$ is invertible. 
If we set $u(t)=\Delta y(t)$ in \eqref{eq:WN-0}, we obtain that for $\tau\in\NN$ and $\zeta\in[0,1]$,
    \begin{align*}
        u_d(\tau)(\zeta)=u((\tau+\zeta)p(1))=\Delta y((\tau+\zeta)p(1))=\Delta y_d(\tau)(\zeta).
    \end{align*}
    Thus, the feedback $u(t)=\Delta y(t)$ implies that $u_d(\tau)=M_\Delta y_d(\tau)$ in \eqref{eq:WN-disc} and vice versa. Thus the solutions of \eqref{eq-WN-0-pert} and \eqref{eq:WN-disc-pert} are related in the sense of Proposition \ref{satz:Diskretisierung}. In Remark \ref{rem:stateRel}, it has been highlighted that, for $\tau\in\NN$ and $\zeta\in[0,1]$, $x_d(\tau)(\zeta) = \lambda_0\left(k^{-1}(\zeta)\right) \tilde{x}\left(k^{-1}(\zeta), \tau p(1)\right)$. Thus, asymptotic or exponential stability of \eqref{eq-WN-0-pert} implies asymptotic or exponential stability of \eqref{eq:WN-disc-pert}, respectively, where we have used that $\Nl (\lambda_0 \circ k^{-1})(\cdot\circ k^{-1} )\Nr_{\L^2([0,1]; \KK^n)}$ is an equivalent norm on $\L^2([0,1];\KK^m)$. For the converse direction we use the identity \begin{align*}
       \tilde{x}(\xi,t)    =\begin{cases}
     \lambda_0(\xi)^{-1} x_d(\tau)(k(\xi)+\delta),& k(\xi)+\delta\leq1, \\
      \lambda_0(\xi)^{-1} x_d(\tau+1)(k(\xi)+\delta-1),& k(\xi)+\delta>1.
    \end{cases}
    \end{align*} for $t=p(1)(\tau+\delta)>0$ with $\tau\in\NN_0$, $\delta\in[0,1)$ and $\xi\in[0,1]$ from Remark \ref{rem:stateRel}.
    As $\Nl\lambda_0^{-1}(\cdot\circ k)\Nr_{\L^2([0,1]; \KK^n)}$ is an equivalent norm, asymptotic or exponential stability of \eqref{eq:WN-disc-pert} implies asymptotic or exponential stability of \eqref{eq-WN-0-pert}, respectively. 
\end{proof}

\section{Some properties of linear finite-dimensional discrete-time systems}\label{sec:IntrFinDimSyst}
We aim to  examine the infinite-dimensional system \eqref{eq:WN-disc} via the finite-dimensional system
\begin{align}\tag{\text{$\Sigma_f$}}\label{eq:fin-dim-syst}
\begin{split}
x_f(\tau+1)&= A x_f(\tau) +B u_f(\tau), \quad \tau \in \NN_0, \quad x_f(0)=x_{f,0}, \\
y_f(\tau)&=C x_f(\tau) + D u_f(\tau), \quad \tau \in \NN_0, \end{split}
\end{align} with initial value $x_{f,0}\in \KK^n$, inputs $u_f(0), u_f(1),\ldots \in \KK^m$ and output space  $\KK^k$. In this section $A\in\KK^{n\times n}, B\in\KK^{n\times m}, C\in\KK^{k\times n}$ and $D\in\KK^{k\times m}$ are arbitrary matrices and 
 we summerize some general basic properties on finite-dimensional discrete-time systems. Before diving into those properties, we recall the expression for the solution of \eqref{eq:fin-dim-syst}, which, for $\tau\in\NN_0$, is given by
\begin{align*}
x_f(\tau)=& A^\tau x_{f,0}+\sum_{\sigma=0}^{\tau-1} A^{\tau-1-\sigma}B u_f(\sigma), \quad 
y_f(\tau)= CA^\tau x_{f,0}+\sum_{\sigma=0}^{\tau-1} CA^{\tau-1-\sigma}B u_f(\sigma)+D u_f(\tau),
\end{align*} 
with initial value $x_{f,0}\in\KK^n$ and inputs $u_f(0),\ldots ,u_f(\tau-1)\in \KK^m$.
In the following, we show some results on the robustness of the stability for \eqref{eq:fin-dim-syst}. We start with the feedback approach and show the well-posedness.
\begin{lemma}
    Let $\Delta\in\KK^{m\times k}$ be such that $I-\Delta D$ is invertible. Then the closed-loop system composed of \eqref{eq:fin-dim-syst} and $u_f(\tau)=\Delta y_f(\tau)$, $\tau\in\NN_0$, is well-posed, that is, it admits a unique solution.
\end{lemma}

    \begin{proof}
        It suffices to observe that the solution of the closed-loop system is given by
         \begin{align*}
             x_f(\tau)&=\left( A+ B(I- \Delta  D)^{-1}  \Delta  C\right)^\tau x_{f,0},\quad \tau \in \NN_0.\qedhere
         \end{align*}
    \end{proof}

\begin{remark}
 The well-posed closed-loop system composed of \eqref{eq:fin-dim-syst} and $u_f(\tau)=\Delta y_f(\tau)$ could be understood as the structured perturbation of \eqref{eq:fin-dim-syst}, which reads as \begin{align}\tag{\text{$\Sigma_f^\Delta$}}\label{eq:fin-dim-syst-pert}
        \begin{split}
            x_f(\tau+1)&= \left( A+ B(I- \Delta  D)^{-1}  \Delta  C\right) x_f(\tau), \quad \tau\in\NN_0,\quad x_f(0)=x_{f,0}.
        \end{split}
    \end{align}
\end{remark}

The stability concepts for \eqref{eq:fin-dim-syst-pert} is given in the next definition.
\begin{definition}
     Let $\Delta\in\KK^{m\times k}$ with $I-\Delta D$ invertible.
     The system \eqref{eq:fin-dim-syst-pert} is said to be \emph{stable} if $\lim_{\tau\to\infty,\tau\in\NN} x_f(\tau)\allowbreak =0$ holds for every initial state $x_{f,0}$.
\end{definition}
This stability concept may be understood as asymptotic stability. In the finite-dimensional case this is equivalent to exponential stability, see e.g.~\cite[Sec.~3.3]{HinD2005b}.

\begin{proposition}\label{prop:FinStab}
Let $\Delta\in\KK^{m\times k}$ be such that $I-\Delta D$ is invertible.  
The following statements are equivalent:

\begin{itemize}
    \item \eqref{eq:fin-dim-syst-pert} is stable.
    \item There exist $\omega>0$ and $M\geq1$ such that $\Vert x_f(\tau) \Vert_{\KK^n} \leq M \mathrm{e}^{-\omega \tau}  \Vert x_{f,0} \Vert_{\KK^n}$ for all $\tau\in\NN_0$.
\item $\sigma\left(A+B(I-\Delta D)^{-1}\Delta C\right)\subset \DD$.
\end{itemize}
\end{proposition}

\eqref{eq:fin-dim-syst} is said to be \emph{stable}, if the state of \eqref{eq:fin-dim-syst} with $u_f(\tau)=0$  tends to zero as $\tau\to\infty$.
According to Proposition \ref{prop:FinStab} (with $\Delta=0$) it follows that \eqref{eq:fin-dim-syst} is stable if and only if $\sigma(A)\subset \DD$. The stability radius is defined as follows:
\begin{definition}\label{def:stab_rad}
        The stability radius of a stable system \eqref{eq:fin-dim-syst} is defined by
    \begin{align*}
        r_\KK\eqref{eq:fin-dim-syst}&:= \inf \left\lbrace \Nl \Delta \Nr_{\KK^{m\times k}}  : \Delta \in \KK^{m\times k}, I-\Delta D \text{ is not invertible or }\eqref{eq:fin-dim-syst-pert} \text{ is not stable} \right\rbrace.
    \end{align*}
\end{definition}

Research on the stability radius for finite-dimensional systems has been initiated in the 80's and 90's. Important results on that topic can be found in e.g.~the notes and references of \cite[Sec.~5.3]{HinD2005b}. We summarize the main results of \cite[Sec.~5.3]{HinD2005b} in the following proposition.
\begin{proposition}\label{prop:StabRadiFin}
   Let $G(z)=C(z-A)^{-1} B+D$, $z\in \rho(A)$ be the so-called transfer function of \eqref{eq:fin-dim-syst} and let \eqref{eq:fin-dim-syst} be stable. The complex stability radius is given by
     \begin{align*}
         r_\CC\eqref{eq:fin-dim-syst}= \left[{\max_{z\in\partial \DD} \Nl G(z) \Nr_{\CC^{k\times m}} }\right]^{-1}.
     \end{align*}

For the real stability radius we  assume that all norms are induced by the Euclidean vector-norm.
  In the case where $m=1$ or $k=1$, the real stability radius takes the form
\begin{align*}
   r_\RR\eqref{eq:fin-dim-syst}= \left[ \max \left\lbrace\Nl D\Nr_{\KK^{k\times m}} , \max_{z\in\partial\DD} \dist\left( \Re (G(z)), \RR \Im (G(z)) \right) \right\rbrace \right]^{-1}
\end{align*}
    and when $m,k>1$ we obtain
          \begin{align*}
         r_\RR\eqref{eq:fin-dim-syst}=\left[{\max\left\lbrace \Nl D \Nr_{\KK^{k\times m}} , \max_{z\in\partial \DD}\inf_{\gamma\in (0,1]} \sigma_2\left(\left[\begin{smallmatrix}
        \Re(G(z)) & -\gamma \Im(G(z)) \\
        \frac{1}{\gamma}\Im(G(z)) & \Re(G(z))
    \end{smallmatrix}\right]\right) \right\rbrace}\right]^{-1}
     \end{align*}
    where $\sigma_2(J)$ stands for the second largest singular value of a matrix $J$. Note that we adopted the convention $\frac{1}{\infty}=0$. 
\end{proposition}

\begin{remark}\label{rem:MinNormPertforFiniteDimSyst}
       In the case where $r_\KK\eqref{eq:fin-dim-syst}$ is finite, the infimum in Definition \ref{def:stab_rad} is attained, and a destabilizing perturbation $\Delta_0$ of minimal norm can be constructed.
        First we shall show that $r_\KK\eqref{eq:fin-dim-syst}\leq\Nl D\Nr_{\KK^{k\times m}}^{-1}$. Before going further, we recall that, for a vector $u\in\KK^m$, we define its norm by $\Vert u\Vert_{\KK^m} = \sqrt{u^*u}$ and the definition used for the norm of a matrix $D\in\KK^{k\times m}$ is $\Vert D\Vert_{\KK^{k\times m}} = \sup_{\Vert u\Vert_{\KK^m} = 1}\Vert Du\Vert_{\KK^k}$. To prove the inequality $r_\KK\eqref{eq:fin-dim-syst}\leq\Nl D\Nr_{\KK^{k\times m}}^{-1}$, assume that $D\neq 0$
        and choose a unitary vector $u_0\in\KK^m$ which maximizes $\Vert Du_0\Vert_{\KK^k}$. For this $u_0$, we define $y_0\in\KK^k$ by $y_0 := \Vert D\Vert^{-1}_{\KK^{k\times m}}Du_0$. In this way, $y_0$ satisfies $\Vert y_0\Vert_{\KK^k} = \Vert D\Vert^{-1}_{\KK^{k\times m}}\Vert Du_0\Vert_{\KK^k} = \Vert D\Vert^{-1}_{\KK^{k\times m}}\Vert D\Vert_{\KK^{k\times m}} = 1$. From the definition of $y_0$, we obtain $Du_0 = \Vert D\Vert_{\KK^{k\times m}}y_0$ which entails $\Vert D\Vert_{\KK^{k\times m}} = y_0^*Du_0$. Then we define $\Delta_0= \Nl D\Nr_{\KK^{k\times m}}^{-1}  u_0y_0^\ast$. We have that $\Vert u_0 y_0^*\Vert_{\KK^{m\times k}} = \sup_{\Vert y\Vert_{\KK^k} = 1}\Vert u_0 y_0^*y\Vert_{\KK^{m}} = \sup_{\Vert y\Vert_{\KK^k} = 1}\vert y_0^*y\vert_{\KK}\leq 1$, where the supremum is reached for $y = y_0$. Now observe that $\Vert \Delta_0\Vert_{\KK^{m\times k}} = \Vert D\Vert_{\KK^{k\times m}}^{-1}\sup_{\Vert y\Vert_{\KK^k} = 1}\Vert u_0 y_0^*y\Vert_{\KK^{m}} = \Vert D\Vert_{\KK^{k\times m}}^{-1}$. Moreover, it holds that $\Delta_0 Du_0 = \Vert D\Vert_{\KK^{k\times m}}^{-1} u_0 y_0^*Du_0 = \Vert D\Vert_{\KK^{k\times m}}^{-1} u_0 \Vert D\Vert_{\KK^{k\times m}} = u_0$, which implies that $I-\Delta_0 D$ is not invertible.

        Therefore by Definition \ref{def:stab_rad} we have $r_\KK\eqref{eq:fin-dim-syst}\leq\Nl D\Nr_{\KK^{k\times m}}^{-1}$. Thus, in the case $r_\KK\eqref{eq:fin-dim-syst}=\Nl D\Nr_{\KK^{k\times m}}^{-1}<\infty$, $\Delta_0$ is a minimal norm destabilizing perturbation of rank 1. In the case where $r_\KK\eqref{eq:fin-dim-syst}<\Nl D\Nr_{\KK^{k\times m}}^{-1}$ we can construct a destabilizing perturbation $\Delta_0$ in a similar way as before. The procedure is as follows, see e.g.~\cite[Sec.~5.3]{HinD2005b}:
        \begin{itemize}
        \item Let $z_0\in\partial\DD$, $u_0\in\CC^m$ and $y_0\in\CC^k$ with $\Nl u_0\Nr_{\CC^m}=\Nl y_0\Nr_{\CC^k}=1$ such that $\max_{z\in\partial \DD} \Nl G(z) \Nr_{\KK^{k\times m}}=\Nl G(z_0)\Nr_{\CC^{k\times m}}= y_0^\ast G(z_0)u_0>0$. Then $\Delta_0= \Nl G(z_0)\Nr_{\CC^{k\times m}}^{-1}u_0 y_0^\ast$ is a minimum norm destabilization of rank 1 of \eqref{eq:fin-dim-syst-pert}.
        \item For $m=1$ let $z_0\in\partial\DD$ and $y_0\in\RR^k$ such that we have \begin{align*}
            \max_{z\in\partial\DD} \dist\left( \Re (G(z)), \RR \Im (G(z)) \right)= \dist\left( \Re (G(z_0)), \RR \Im (G(z_0)) \right)>\Nl D \Nr_{\KK^{k\times m}} 
        \end{align*}with $y_0^\T\Re (G(z_0))=\dist\left( \Re (G(z_0)), \RR \Im (G(z_0)) \right)$ and $y_0^\T \Im (G(z_0))=0$. Then 
        \begin{align*}
        \Delta_0= \left[\dist\left( \Re (G(z_0)), \RR \Im (G(z_0)) \right)\right]^{-1}y_0^\T
        \end{align*}
        is a minimum norm destabilizing real perturbation for \eqref{eq:fin-dim-syst-pert}. The case $k=1$ can be treated similarly.
         \item For the case $m,k>1$ a minimum norm real destabilization can also be constructed, which is not necessarily of rank 1, see e.g.~\cite[Sec.~5.3]{HinD2005b}.
    \end{itemize}
\end{remark}

\section{Finite-dimensional approach for the network of waves}\label{sec:FinAppr}
In this section, we derive our main results for \eqref{eq:WN-disc}, building on analogous results obtained for the finite-dimensional system \eqref{eq:fin-dim-syst}. In particular, we show that the stability of \eqref{eq:WN-P0-pert} can be inferred from the stability of the finite-dimensional discrete-time system \eqref{eq:fin-dim-syst-pert}.
We start with the following proposition at the discrete-time level.
\begin{proposition}\label{prop:Stab-Dis-Fin}
    Let $\Delta\in\KK^{m \times k}$ with $I-\Delta D$ invertible. The following statements are equivalent: \begin{enumerate}
        \item The infinite-dimensional system \eqref{eq:WN-disc-pert} is exponentially stable.
        \item The infinite-dimensional system \eqref{eq:WN-disc-pert} is asymptotically stable.
        \item The finite-dimensional system \eqref{eq:fin-dim-syst-pert} is stable.
    \end{enumerate}
\end{proposition}

   \begin{proof}
        The implication "1.$\Rightarrow$2." is obvious. To show that "2.$\Rightarrow$3.", for an arbitrary initial state $x_{f,0}\in\KK^n$ we consider the constant function $x_{d,0}:= x_{f,0}\in\L^2([0,1];\KK^n)$. From the assumption, the solution $x_d(\tau)$ of \eqref{eq:WN-disc-pert} with initial state $x_{d,0}$ tends to zero as $\tau\to\infty$. Observe now that the solution of \eqref{eq:fin-dim-syst-pert} at time $\tau$ with initial state $x_{f,0}$, denoted by $x_f(\tau)$, is given by $x_{d}(\tau) = x_{f}(\tau)$. Therefore $x_f(\tau)$ tends to zero as $\tau\to\infty$. Finally, we focus on the implication "3.$\Rightarrow$1.". The solution of \eqref{eq:fin-dim-syst-pert} is given by $x_f(\tau)=\left( A+B(I-\Delta D)^{-1}\Delta C\right)^\tau x_{f,0}$. Stability of \eqref{eq:fin-dim-syst-pert} and Proposition \ref{prop:FinStab} imply that $\Nl x_f(\tau) \Nr_{\KK^n} \leq M \mathrm{e}^{-\omega \tau}  \Nl x_{f,0} \Nr_{\KK^n}$ for some constants $M\geq 1$ and $\omega>0$.
     Thus, \begin{align*}
           \Nl \left(M_A+M_B(I-M_\Delta M_D)^{-1} M_\Delta M_C\right)^\tau\Nr_{\mathcal{L}(\L^2([0,1]; \KK^n))}=&\Nl M_{\left(A+B(I-\Delta D)^{-1}\Delta C\right)^\tau} \Nr_{\mathcal{L}(\L^2([0,1]; \KK^n))} \\
          =&\Nl \left(A+B(I-\Delta D)^{-1}\Delta C\right)^\tau \Nr_{\KK^{n\times n}} \\
          \leq & M \mathrm{e}^{-\omega \tau}.\qedhere
     \end{align*}
    \end{proof}

The previous proposition is not limited to systems with matrices $A, B, C$ and $D$ as defined in \eqref{eq:Mat}, it is valid for arbitrary matrices $A, B, C$ and $D$. Proposition~\ref{prop:Stab-Dis-Fin} allows us to state the following theorem that shows the equivalence between the stability radii of the network of waves and the associated finite-dimensional system.
\begin{theorem}\label{thm:StabRadiEq}
Let $\Delta\in\KK^{m \times k}$ and assume that $I-\Delta D$ is invertible.
For the perturbed system \eqref{eq:WN-P0-pert}, the following statements are equivalent:
\begin{itemize}
\item it is asymptotically stable,
\item it is exponentially stable,
\item the finite-dimensional system \eqref{eq:fin-dim-syst-pert} is stable.
\end{itemize}

If the system \eqref{eq:WN-P0} (equiv.~\eqref{eq:fin-dim-syst}) is stable, then it holds that $r_\KK\eqref{eq:WN-P0}=  r_\KK\eqref{eq:fin-dim-syst}$.
\end{theorem}

\begin{proof}
    The stability results are obtained by invoking Proposition \ref{prop:Stab-Dis-Fin} and Lemma \ref{lem:Stab-Wn-Disc}. From that, we get that the stability radii of \eqref{eq:WN-P0} and \eqref{eq:fin-dim-syst} coincide.
\end{proof}

Theorem \ref{thm:StabRadiEq} enables us to prove our main results.
\begin{proof}[Proof of Proposition \ref{prop:MainresStab}]
The proof follows from Theorem \ref{thm:StabRadiEq} and the matrix condition follows from Proposition \ref{prop:FinStab}.
\end{proof}
\begin{proof}[Proof of Theorem \ref{thm:MainresRob}]
The stability formulas are obtained by Theorem \ref{thm:StabRadiEq} and Proposition \ref{prop:StabRadiFin}.
\end{proof}

\begin{remark}\label{rem:MinNormPertINFITY}
   If $r_\KK\eqref{eq:WN-P0}$ is finite Theorem \ref{thm:StabRadiEq} enables us to use the perturbations constructed in Remark \ref{rem:MinNormPertforFiniteDimSyst} for \eqref{eq:fin-dim-syst} to destabilize \eqref{eq:WN-P0-pert} with minimal norm.
\end{remark}

\section{The dual system and its application to controllability and observability}\label{sec:DualAndContrObs}
This section is dedicated to calculate the dual system of \eqref{eq:WN-0} (equiv.~\eqref{eq:WN-P0}). Furthermore, we apply the discrete-time approach used for the computation of the stability radii to the analysis of controllability and observability. Starting from the discrete-time formulation, the dual system associated with \eqref{eq:WN-disc} is readily seen to be given by
\begin{align}\tag{\text{$\Sigma_d^\ast$}}\label{eq:WN-disc-dual}\begin{split}
x_d(\tau+1) &= M_{A}^\ast x_d(\tau) + M_{C}^\ast y_d(\tau),\quad \tau\in\NN_0,\quad x_d(0) = x_{d,0},\\
u_d(\tau) &= M_{B}^\ast x_d(\tau) + M_{D}^\ast y_d(\tau), \quad \tau\in\NN_0,
\end{split}
\end{align} with input $y_d(\tau)\in\L^2([0,1];\KK^k)$ and output $u_d(\tau)\in\L^2([0,1];\KK^m)$, $\tau\in\NN_0$.
    From Proposition \ref{satz:Diskretisierung} we see that \eqref{eq:WN-disc-dual} is equivalent to the system 
    \begin{align}\tag{\text{$\Sigma_0^\ast$}}\label{eq:WN_Dual_from_Disc}
\begin{split}
\frac{\partial \hat{x}}{\partial t}(\xi, t) &= -\frac{\partial}{\partial \xi} \left( \lambda_0(\xi)\hat{x}(\xi, t)\right),   \quad t\geq 0, \xi\in [0,1], \quad \hat{x}(\xi,0) = \hat{x}_0(\xi),  \quad\xi\in [0,1], \\
C^\ast y(t)&=\lambda_0(0)\hat{x}(0, t)-A^\ast\lambda_0(1)\hat{x}(1,t),  \quad t\geq 0,\\
u(t)&=B^\ast\lambda_0(1)\hat{x}(1, t)+D^\ast y(t), \quad t\geq 0,
\end{split}
\end{align}
where $y(t)\in\KK^k$ is the input and $u(t)\in\KK^m$ is the output, $t\geq 0$.

\begin{remark}\label{rem:dual} 
Before going further in the verification that \eqref{eq:WN_Dual_from_Disc} is indeed the dual of \eqref{eq:WN-0}, we shall express \eqref{eq:WN_Dual_from_Disc} equivalently. For this, consider the function $r:[0,1]\to [0,1]$ defined by
\begin{equation*}
    r(\xi) = p^{-1}(p(1)-p(\xi)),
\end{equation*}
where $p$ is defined in \eqref{lem:Fctkp} and $p^{-1}$ stands for its inverse. According to the expression of $p$, it is easy to see that $r(0) = 1$ and $r(1) = 0$. Moreover, it holds that $p(r(\xi)) = p(1)-p(\xi)$, which, after differentiation, implies that $r'(\xi) = -\frac{\lambda_0(r(\xi))}{\lambda_0(\xi)}$. Now consider the change of variables
\begin{equation}
    {\bf x}(\xi,t) = \frac{\lambda_0(r(\xi))}{\lambda_0(\xi)}\hat{x}(r(\xi),t).
    \label{eq:chang_var_dual}
\end{equation}
With this change of variables, we obtain $\lambda_0(0){\bf x}(0,t) = \lambda_0(1)\hat{x}(1,t)$ and $\lambda_0(1){\bf x}(1,t) = \lambda_0(0)\hat{x}(0,t)$
Moreover, by differentiating ${\bf x}(\xi,t)$ with respect to $t$, we get that
\begin{align*}
    \frac{\partial{\bf x}}{\partial t}(\xi,t) = \frac{\lambda_0(r(\xi))}{\lambda_0(\xi)} \frac{\partial\hat{x}}{\partial t}(r(\xi),t) = -\frac{\lambda_0(r(\xi))}{\lambda_0(\xi)}\frac{\partial}{\partial\xi}(\lambda_0(r(\xi))\hat{x}(r(\xi),t)) &=r'(\xi)\frac{\partial}{\partial\xi}(\lambda_0(r(\xi))\hat{x}(r(\xi),t))\\
    &= \frac{\partial}{\partial\xi}(\lambda_0(\xi){\bf x}(\xi,t)).
\end{align*}
This implies that \eqref{eq:WN_Dual_from_Disc} is also equivalent to the following system
\begin{align}\tag{\text{$\Sigma_0^\ast$}}\label{eq:WN_Dual_from_Dual}
\begin{split}
\frac{\partial {\bf x}}{\partial t}(\xi, t) &= \frac{\partial}{\partial \xi} \left( \lambda_0(\xi){\bf x}(\xi, t)\right),   \quad t\geq 0, \xi\in [0,1], \quad {\bf x}(\xi,0) =: {\bf x}_0(\xi),  \quad\xi\in [0,1], \\
C^\ast y(t)&=\lambda_0(1){\bf x}(1, t)-A^\ast\lambda_0(0){\bf x}(0,t),  \quad t\geq 0,\\
u(t)&=B^\ast\lambda_0(0){\bf x}(0, t)+D^\ast y(t), \quad t\geq 0,
\end{split}
\end{align}
where ${\bf x}_0(\xi) = \frac{\lambda_0(r(\xi))}{\lambda_0(\xi)}\hat{x}_0(r(\xi))$. By defining the operator $R:\L^2([0,1];\KK^n)\to\L^2([0,1];\KK^n)$ as $(R\hat{x})(\xi) = \frac{\lambda_0(r(\xi))}{\lambda_0(\xi)}\hat{x}(r(\xi))$ where we equip $\L^2([0,1];\KK^n)$ with the inner product $\langle f,g\rangle_{\L^2([0,1];\KK^n)} := \int_0^1 g^*(\xi)\lambda_0(\xi)f(\xi)\d\xi$, we have that 
\begin{align*}
    \Vert R\hat{x}\Vert_{\L^2([0,1];\KK^n)}^2 = \int_0^1\frac{\lambda_0(r(\xi))^2}{\lambda_0(\xi)}\Vert \hat{x}(r(\xi))\Vert_{\KK^n}^2\d\xi = -\int_0^1 \lambda_0(r(\xi))\Vert\hat{x}(r(\xi))\Vert^2_{\KK^n}r'(\xi)\d\xi &= \int_0^1 \lambda_0(\eta)\Vert\hat{x}(\eta)\Vert^2_{\KK^n}\d\eta\\
    &= \Vert\hat{x}\Vert_{\L^2([0,1];\KK^n)}^2.
\end{align*}
In particular, this implies that $\Vert{\bf x}(\cdot,t)\Vert_{\L^2([0,1];\KK^n)} = \Vert\hat{x}(\cdot,t)\Vert_{\L^2([0,1];\KK^n)}$ for all $t\geq 0$, emphasizing the equivalence between \eqref{eq:WN_Dual_from_Disc} and \eqref{eq:WN_Dual_from_Dual}. It is also worth noting that in the case where $\lambda_0$ is constant, the function $r$ is given by $r(\xi) = 1-\xi$. Then, the change of variables \eqref{eq:chang_var_dual} is simply ${\bf x}(\xi,t) = \hat{x}(1-\xi,t)$, which consists in looking at the spatial variable in the reverse direction.
\end{remark}

One way to verify that \eqref{eq:WN_Dual_from_Disc} is indeed the dual of \eqref{eq:WN-0} is to use the system node approach, which provides an elegant framework for this verification.
 We note  that \eqref{eq:WN-0} defines a system node $S:D(S)\subset X\times\KK^m\to X\times \KK^k$,  given by
        \begin{align*}
            S\left(\begin{smallmatrix}
                \tilde{x}\\ u 
                \end{smallmatrix}\right)&=\left(\begin{smallmatrix}
                    -\frac{\d}{\d \xi}(\lambda_0\tilde{x}) \\ C(\lambda_0\tilde{x})(1)+D u
                \end{smallmatrix} \right),\\
                D(S)&=\left\{\left[\begin{smallmatrix}\tilde{x}\\ u\end{smallmatrix}\right]\in Z\times\mathbb{K}^m: (\lambda_0 \tilde{x})(0) - A(\lambda_0 \tilde{x})(1)= B u\right\},
        \end{align*}
        see~\cite[Sec.~3]{HasA2025a2}. Here, $X$ is the Hilbert space $\L^2([0,1];\KK^n)$ equipped with $\Sl f,g \Sr_X:= \int_0^1 g^\ast \lambda_0 f\d\xi$ and $Z$ is defined as in Definition \ref{defi:well-posed}.
        The dual of \eqref{eq:WN-0} may be defined as the system which corresponds to the system node $S^\ast$, see~\cite[Thm.~3.5]{StafO2004a}. Here, by $S^\ast$, we mean the Hilbert space adjoint of $S$, whose domain is defined by
        \begin{equation}
        D(S^*)=\left\{b\in X\times \KK^k: \exists c\geq 0 \text{ such that } \vert\langle Sa,b\rangle\vert\leq c\Vert a\Vert_{X\times\KK^m}\,\, \forall a\in D(S) \right\}
        \label{eq:D(S*)}.
        \end{equation}
        If we pick $b\in D(S^*)$, $S^*b$ is the unique element $q\in X\times \KK^m$ such that $\Sl Sa,b \Sr = \Sl a, q\Sr$ for all $a\in D(S)$. Our main result on the dual of \eqref{eq:WN-0} is stated in the next theorem.
\begin{theorem}\label{prop:DualViaSysNode}
    The dual system of \eqref{eq:WN-0} is given by \eqref{eq:WN_Dual_from_Disc}. In particular, the adjoint of the system node $S$ corresponding to \eqref{eq:WN-0} is given by 
    \begin{align}
            &S^\ast\left(\begin{smallmatrix}
                \hat{x}\\ y 
                \end{smallmatrix}\right)=\left(\begin{smallmatrix}
                    \frac{\d}{\d \xi}(\lambda_0\hat{x}) \\ B^\ast(\lambda_0\hat{x})(0)+D^\ast y
                \end{smallmatrix} \right),\label{eq:S_star}\\
                &D(S^*)=\left\{\left[\begin{smallmatrix}\hat{x}\\ y\end{smallmatrix}\right]\in Z\times\mathbb{K}^k: (\lambda_0 \hat{x})(1)=A^\ast(\lambda_0 \hat{x})(0) + C^\ast y\right\},\label{eq:H}
        \end{align} which is the system node associated to the dual system of \eqref{eq:WN-0}.
\end{theorem}

\begin{proof}
        Denote the right-hand side of \eqref{eq:H} by $H$. We start by showing that $D(S^*)\subseteq H$ and that \eqref{eq:S_star} holds.
       In order to calculate the adjoint of $S$, let $\left( \begin{smallmatrix}
            \tilde{x}\\u
        \end{smallmatrix}\right)\in D(S)$ and let $\left( \begin{smallmatrix}
            \hat{x}\\  y 
        \end{smallmatrix}\right)\in X\times \KK^k$. We have that\begin{align}\label{eq:CalcAdjFirstPart}
            \Sl S\left( \begin{smallmatrix}
            \tilde{x}\\u
        \end{smallmatrix}\right),  \left( \begin{smallmatrix}
            \hat{x}\\  y 
        \end{smallmatrix}\right)\Sr_{X\times \KK^k}= -\int_0^1 \lambda_0\hat{x}^\ast \frac{\d}{\d\xi}(\lambda_0\tilde{x}) \d\xi + y^\ast C(\lambda_0\tilde{x})(1)+y^\ast D u.
        \end{align}
        It holds that $\left( \begin{smallmatrix}
            \hat{x}\\  y 
        \end{smallmatrix}\right)\in D(S^\ast)$ if and only if there exists $\left( \begin{smallmatrix}
            x'\\  u' 
        \end{smallmatrix}\right)\in X\times \KK^m$ such that $\Sl S\left( \begin{smallmatrix}
            \tilde{x}\\u
        \end{smallmatrix}\right),  \left( \begin{smallmatrix}
            \hat{x}\\  y 
        \end{smallmatrix}\right)\Sr_{X\times \KK^k}= \Sl \left( \begin{smallmatrix}
            \tilde{x}\\u
        \end{smallmatrix}\right),  \left( \begin{smallmatrix}
            x'\\  u' 
        \end{smallmatrix}\right)\Sr_{X\times \KK^m}$ for every $\left( \begin{smallmatrix}
            \tilde{x}\\u
        \end{smallmatrix}\right)\in D(S)$. If we consider $\left( \begin{smallmatrix}
            \tilde{x}\\u
        \end{smallmatrix}\right)\in D(S)$ with $u=0$ and $\tilde{x}(1)=0$ we obtain with \eqref{eq:CalcAdjFirstPart} that $-\int_0^1 \lambda_0\hat{x}^\ast \frac{\d}{\d\xi}(\lambda_0\tilde{x}) \d\xi= \int x'^\ast \lambda_0\tilde{x}\d\xi$ holds for every such $\tilde{x}$. Therefore we have $\lambda_0\hat{x}\in\H^1([0,1];\KK^n)$ and especially $D(S^\ast)\subset Z$.
        Going back to \eqref{eq:CalcAdjFirstPart} we can now assume that $\lambda_0\hat{x}\in \H^1([0,1];\KK^n)$. By using integration by parts and the equality $(\lambda_0 \tilde{x})(0) - A(\lambda_0 \tilde{x})(1)= B u$, we get \begin{align}
             &\Sl S\left( \begin{smallmatrix}
            \tilde{x}\\u
        \end{smallmatrix}\right),  \left( \begin{smallmatrix}
            \hat{x}\\  y 
        \end{smallmatrix}\right)\Sr_{X\times \KK^k}\nonumber \\
       =&
        (\lambda_0\hat{x}^\ast)(0) (\lambda_0\tilde{x})(0)-(\lambda_0\hat{x}^\ast)(1) (\lambda_0\tilde{x})(1)+ \int_0^1 \left(\frac{\d}{\d\xi}(\lambda_0\hat{x}^\ast)\right) \lambda_0\tilde{x} \d\xi + y^\ast C(\lambda_0\tilde{x})(1)+y^\ast D u\nonumber\\
         =& 
        [A^\ast (\lambda_0\hat{x})(0) -(\lambda_0\hat{x})(1)   + C^\ast y ]^\ast(\lambda_0\tilde{x})(1) +[ B^\ast (\lambda_0\hat{x})(0)+ D^\ast y ]^\ast u 
        + \int_0^1 \left(\frac{\d}{\d\xi}(\lambda_0\hat{x})\right)^\ast \lambda_0\tilde{x} \d\xi.\label{eq:Int_Adjoint}
        \end{align}
        Since the mapping $\tilde{x}\mapsto (\lambda_0\tilde{x})(1)$ is not bounded on $X$, we conclude with the definition of $D(S^*)$, see \eqref{eq:D(S*)}, that $\left( \begin{smallmatrix}
            \hat{x}\\  y 
        \end{smallmatrix}\right)\in D(S^\ast)$
        if and only if $A^\ast (\lambda_0\hat{x})(0) -(\lambda_0\hat{x})(1)   + C^\ast y =0$. Thus we get that the adjoint operator is given by \eqref{eq:S_star} and we obtain $D(S^*)\subseteq H$. The inclusion $H\subseteq D(S^*)$ follows directly from the definition of $H$ and from \eqref{eq:Int_Adjoint}. 
Finally, we observe that $(S^*,D(S^*))$ defined in \eqref{eq:S_star} and \eqref{eq:H} corresponds to
\begin{align}\label{eq:WN-0-Dual}
\begin{split}
\frac{\partial \hat{x}}{\partial t}(\xi, t) &= \frac{\partial}{\partial \xi} \left( \lambda_0(\xi)\hat{x}(\xi, t)\right),   \quad t\geq 0, \xi\in [0,1], \quad \hat{x}(\xi,0)= \hat{x}_0(\xi),  \quad\xi\in [0,1], \\
C^\ast y(t)&=\lambda_0(1)\hat{x}(1, t)-A^\ast\lambda_0(0)\hat{x}(0,t),  \quad t\geq 0,\\
u(t)&=B^\ast\lambda_0(0)\hat{x}(0, t)+D^\ast y(t), \quad t\geq 0,
\end{split}
\end{align} with input $y(t)\in\KK^k$ and output $u(t)\in\KK^m$, $t\geq 0$. 
The proof concludes by noting that \eqref{eq:WN-0-Dual} is exactly \eqref{eq:WN_Dual_from_Dual}.
    \end{proof}
 
An illustration of the previous theorem can be found in Figure \ref{fig:propDual}.
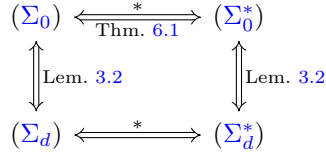
\begin{figure}
\begin{center}
{\normalsize
\begin{tikzcd}[row sep=large,column sep=huge]
\eqref{eq:WN-0} \arrow[r, Leftrightarrow, "\text{$*$}","\text{Thm.~\ref{prop:DualViaSysNode}}"'] \arrow[d, Leftrightarrow, "\text{Lem.~\ref{satz:Diskretisierung}}"]
& \eqref{eq:WN_Dual_from_Disc} \arrow[d, Leftrightarrow, "\text{Lem.~\ref{satz:Diskretisierung}}"] \\
\eqref{eq:WN-disc} \arrow[r, Leftrightarrow, "\text{$*$}"]
& \eqref{eq:WN-disc-dual}
\end{tikzcd}}
\end{center}
\caption{Illustration of Theorem \ref{prop:DualViaSysNode}.\label{fig:propDual}}
\end{figure}

The remainder of this section is devoted to applying the duality result to the concepts of controllability and observability.
For infinite-dimensional systems, it is customary to distinguish between exact and approximate controllability/observability, by which we mean the following concepts.
\begin{definition}\label{def:controllability_WN}\begin{itemize}
    \item \eqref{eq:WN-P0} is said to be \emph{exactly controllable} if there exists $t\in[0,\infty)$ such that the system can be steered between two arbitrary states in time $t$, i.e., for all $x_0,x_1\in \L^2([0,1];\KK^n)$ there exists an input function $u\in\L^2([0,t];\KK^m)$ such that $x(\cdot,t)=x_1$. 

 The system \eqref{eq:WN-P0} is called \emph{approximately controllable} if there exists $t\in[0,\infty)$ such that the system can be steered from any state arbitrarily close to any other state in time $t$, i.e., for all $x_0,x_1\in \L^2([0,1];\KK^n)$ and every $\epsilon>0$ there exists $u\in\L^2([0,t];\KK^m)$ such that $\Nl x(\cdot,t)-x_1\Nr_{\L^2([0,1]; \KK^n)}<\epsilon$.

 \item \eqref{eq:WN-P0} is \emph{approximately observable} if there exists $t\in[0,\infty)$ such that the initial condition $x_0$ can be constructed from the output $y$ on $[0,t]$.
    If $x_0$ can be constructed continuously from $y$ on $[0,t]$, \eqref{eq:WN-P0} is \emph{exactly observable}.
\end{itemize}
\end{definition}

It is well-known that controllability and observability are dual concepts, which means the following, see e.g.~\cite[Sec.~11.2]{TucM2009b}.
\begin{proposition}\label{prop:DualInfDimConti}
The system \eqref{eq:WN-0} is exactly or approximately observable, if and only if the dual system \eqref{eq:WN_Dual_from_Disc} is exactly or approximately controllable, respectively.
\end{proposition}

Our main result concerning controllability and observability shows that the two controllability/oberserva\-bility concepts are equivalent for \eqref{eq:WN-P0}. Moreover, this equivalence can be characterized in terms of certain basic matrix conditions.

\begin{theorem}\label{thm:MainresContr}
   Let the matrices $A,B,C$ and $D$ be defined as in \eqref{eq:Mat}. Then the following statements concerning controllability are equivalent:
\begin{itemize}
     \item \eqref{eq:WN-P0} is exactly controllable.
     \vspace{1mm}
        \item \eqref{eq:WN-P0} is approximately controllable.
        \vspace{1mm}
        \item 
$\rank\left[\begin{smallmatrix}
B & AB & A^2B &\cdots &A^{n-1}B 
\end{smallmatrix}\right] = n$.
\vspace{1mm}
      \item $\rank\left[\begin{smallmatrix}
            \lambda-A & B
          \end{smallmatrix}\right] = n$
        for every $\lambda\in\CC$.
        \vspace{1mm}
            \item $\rank\left[\begin{smallmatrix}
           \lambda-A & B
          \end{smallmatrix}\right] = n$
        for every $\lambda\in\sigma(A)$.
\end{itemize}
For the observability we have the equivalence between the following statements:
\begin{itemize}
     \item \eqref{eq:WN-P0} is exactly observable.
        \item \eqref{eq:WN-P0} is approximately observable.
        \vspace{1mm}
        \item $\rank\left[\begin{smallmatrix} C\\ CA\\ \vdots \\ CA^{n-1}  \end{smallmatrix}\right] = n$.
        \vspace{1mm}
      \item $\rank\left[\begin{smallmatrix}
 \lambda-A \\ C
        \end{smallmatrix}\right] = n$
        for every $\lambda\in\CC$.
        \vspace{1mm}
            \item $\rank\left[\begin{smallmatrix}
             \lambda-A \\ C
          \end{smallmatrix}\right]=n$
      for every $\lambda\in\sigma(A)$.
\end{itemize}
\end{theorem}

In \cite[Rem.~3.21]{ChiY2024a}, it is noted that exact and approximate controllability are equivalent for this class of systems, where $P_0$ is assumed to be diagonal. It is further claimed there that matrix tests can be derived by considering associated difference equations. For our class of systems, we provide an alternative proof of this equivalence, based on the results established in Section \ref{sec:DiscreteAppr}. Thanks to our approach, we state concrete matrix tests for controllability. Thanks to Theorem \ref{prop:DualViaSysNode}, we show that both observability concepts are equivalent for \eqref{eq:WN-P0} and can be tested by simple matrix tests. The proof of Theorem \ref{thm:MainresContr} can be found at the end of this section.\footnote{Theorem \ref{thm:MainresContr} is part of the master’s thesis of H. Wagener at the University of Wuppertal, 2025.} Before defining the controllability at the discrete-time level, we define the controllability operator for \eqref{eq:WN-0}.
\begin{remark}
Thanks to the linearity of \eqref{eq:WN-P0}, $x_0$ in Definition \ref{def:controllability_WN} can be chosen as $x_0 = 0$. 
Thus we may characterize the controllability of \eqref{eq:WN-0} by the controllability operator $\mathcal{B}_t^\eqref{eq:WN-0}: \L^2([0,t];\KK^m)\to \L^2([0,1];\KK^n)$ which maps the input function to the state $\tilde{x}(\cdot,t)$ of \eqref{eq:WN-0} at time $t$ with initial value $\tilde{x}_0=0$. Hence, by looking at Definition \ref{def:controllability_WN}, \eqref{eq:WN-0} is said to be exactly controllable if there exists $t>0$, such that $\ran \mathcal{B}_t^\eqref{eq:WN-0}=\L^2([0,1];\KK^n)$ and it is approximately controllable, if $\overline{\ran \mathcal{B}_t^\eqref{eq:WN-0}} = \L^2([0,1];\KK^n)$.
\end{remark}

We define the controllability for the discrete-time system \eqref{eq:WN-disc} by the operator which maps the input sequence $( u_d(0),\ldots, u_d(\tau-1))$ to the state $x_d(\tau)$.
\begin{definition}
    The \emph{controllability operator} of \eqref{eq:WN-disc} at time $\tau\in\NN$, denoted by $\mathcal{B}_\tau^\eqref{eq:WN-disc}:\left(\L^2([0,1];\KK^m)\right)^\tau\to\L^2([0,1];\KK^n)$, is defined by
    \begin{equation*}
    \mathcal{B}_\tau^\eqref{eq:WN-disc}\left( u_d(0),u_d(1),\ldots, u_d(\tau-1) \right) := \sum_{\sigma=0}^{\tau-1} M_A^{\tau-1-\sigma}M_B u_d(\sigma).
    \end{equation*}
    \eqref{eq:WN-disc} is said to be exactly controllable if $\ran\mathcal{B}_\tau^\eqref{eq:WN-disc}=\L^2([0,1];\KK^n)$. If $\ran\mathcal{B}_\tau^\eqref{eq:WN-disc}$ is dense in $\L^2([0,1];\KK^n)$ then \eqref{eq:WN-disc} is said to be approximately controllable.
\end{definition}

Note that the time in the definitions of controllability can be simply enlarged, i.e., it holds $\ran\mathcal{B}_t^\eqref{eq:WN-0}\subseteq \ran\mathcal{B}_s^\eqref{eq:WN-0}$ for $0<t<s$ and $\ran\mathcal{B}_\tau^\eqref{eq:WN-disc}\subseteq \ran\mathcal{B}_\sigma^\eqref{eq:WN-disc}$ for integers $0<\tau<\sigma$. The next lemma shows that our definitions of controllability on the continuous- and discrete-time levels are compatible.

\begin{lemma}\label{lem:ControlContDisc}
    \eqref{eq:WN-0} is exactly or approximately controllable if and only if \eqref{eq:WN-disc} is exactly or approximately controllable, respectively.
    \begin{proof}
        By Remark \ref{rem:stateRel} we have $\tilde{x}(\xi,t)=\lambda_0(\xi)^{-1} x_d\left(p(1)^{-1}t\right)(k(\xi))$ for\footnote{Note that you may enlarge the time $t$, such that $ p(1)^{-1}t \in \NN$.} $p(1)^{-1}t\in \NN$ and $\xi\in[0,1]$. Thus, we get \begin{align*}
                \ran \mathcal{B}_t^\eqref{eq:WN-0}=\left\{\lambda_0^{-1}(f\circ k):f\in  \ran \mathcal{B}_{ p(1)^{-1}t }^\eqref{eq:WN-disc} \right\}.
        \end{align*} Since $k$ is a homeomorphism and $\lambda_0$ satisfies $\lambda_0, \lambda_0^{-1}\in \L^\infty([0,1];(0,\infty))$ we conclude that $\ran \mathcal{B}_t^\eqref{eq:WN-0}$ is equal to (resp. dense in) $\L^2([0,1];\CC^n)$ if and only if $\mathcal{B}_{ p(1)^{-1}t }^\eqref{eq:WN-disc}$ is equal to (resp. dense in) to $\L^2([0,1];\CC^n)$.
    \end{proof}
\end{lemma}

Our aim is to characterize the controllability of \eqref{eq:WN-disc} via the finite-dimensional system \eqref{eq:fin-dim-syst}.
Therefore we introduce the concept of controllability for \eqref{eq:fin-dim-syst}. 

\begin{definition}
   The controllability operator associated to \eqref{eq:fin-dim-syst} at time $\tau\in\NN_0$, denoted by $\mathcal{B}_\tau^\eqref{eq:fin-dim-syst}:(\KK^m)^\tau \to \KK^n$, is defined by $\mathcal{B}_\tau^\eqref{eq:fin-dim-syst} \left(u_f(0),u_f(1),\ldots,u_f(\tau-1)\right):= \sum_{\sigma=0}^{\tau-1} A^{\tau-1-\sigma}B u_f(\sigma)$.
\eqref{eq:fin-dim-syst} is said to be controllable if there exists $\tau\in\NN$ such that $\ran\mathcal{B}_\tau^\eqref{eq:fin-dim-syst}=\KK^n$.
\end{definition}
Since finite-dimensional spaces are closed, we do not have to distinguish between approximate or exact controllability. 
In the next proposition, we show that exact and approximate controllability of \eqref{eq:WN-disc} are equivalent to the controllability of \eqref{eq:fin-dim-syst}.
\begin{proposition}\label{prop:controlConnectFinDis}
    The following statements are equivalent:
\begin{enumerate}
    \item \eqref{eq:WN-disc} is exactly controllable.
    \item \eqref{eq:WN-disc} is approximately controllable.
    \item \eqref{eq:fin-dim-syst} is controllable.
\end{enumerate}
\end{proposition}

\begin{proof}
    "1.$\Rightarrow$3.": Since we assume \eqref{eq:WN-disc} to be exactly controllable, we know that there exists a $\tau\in\NN$ such that $\ran \mathcal{B}_\tau^\eqref{eq:WN-disc}=\L^2([0,1];\KK^n)$. Let $x_f\in\KK^n$ and consider the constant function $x_d = x_f\in\L^2([0,1];\KK^n)$.
    By assumption there exists an input sequence $u_d(0),  u_d(1), \ldots, u_d(\tau-1)\in \L^2([0,1];\KK^m)$, such that we have 
$\mathcal{B}^\eqref{eq:WN-disc}_\tau \left(u_d(0),u_d(1),\ldots,u_d(\tau-1)\right)= \sum_{\sigma=0}^{\tau-1} A^{\tau-1-\sigma}B u_d(\sigma)= x_d$.
 Integration gives 
$\sum_{\sigma=0}^{\tau-1} A^{\tau-1-\sigma}B \int_0^1 u_d(\sigma) \d\lambda \allowbreak = \int_0^1 x_d \d\lambda =x_f$.
Thus by defining $u_f(\sigma):=\int_0^1 u_d(\sigma) \d\lambda\in\KK^m$ for all $\sigma=0,1,\ldots, \tau-1$, we obtain \linebreak
    $\mathcal{B}^\eqref{eq:fin-dim-syst}_\tau \left(u_f(0),\ldots, u_f(\tau-1)\right)=x_f$.
Since $x_f$ is arbitrary, we get $\ran \mathcal{B}^\eqref{eq:fin-dim-syst}_\tau=\KK^n$.

"3.$\Rightarrow$2.": Let us pick $\tau\in\NN$ such that $\ran \mathcal{B}^\eqref{eq:fin-dim-syst}_\tau=\KK^n$. Let $c\in\KK^n$ and consider $S\subset [0,1]$ measurable. Then there exist $u_f(0),u_f(1),\ldots,u_f(\tau-1)\in \KK^m$ such that 
$\mathcal{B}_\tau^\eqref{eq:fin-dim-syst} \left(u_f(0),u_f(1),\ldots,u_f(\tau-1)\right)= \sum_{\sigma=0}^{\tau-1} A^{\tau-1-\sigma}B u_f(\sigma)= c$.
We set $u_d(\sigma):= u_f(\sigma) \chi_S \in \L^2([0,1];\KK^m)$ for $\sigma=0,1,\ldots, \tau-1$, where
    $\chi_S: [0,1] \to \RR, s \mapsto \left\{\begin{smallmatrix} 1,& s\in S, \\ 0, & \text{else}. \end{smallmatrix}\right.$
With this we obtain $\mathcal{B}_\tau^\eqref{eq:WN-disc} \left(u_d(0),u_d(1),\ldots,u_d(\tau-1)\right)=c \chi_S$. Since $c\in\KK^n$ and $S\subset [0,1]$ were arbitrary and $\mathcal{B}_\tau^\eqref{eq:WN-disc}$ is linear, all elementary functions are in $\ran \mathcal{B}_\tau^\eqref{eq:WN-disc}$. Thus $\overline{\ran \mathcal{B}_\tau^\eqref{eq:WN-disc}}=\L^2([0,1];\KK^n)$ by the construction of the Lebesgue-integral.

"2.$\Rightarrow$1.": We show that the controllability operator $\mathcal{B}_\tau^\eqref{eq:WN-disc}$ has a closed range.
For the input sequence $u_d:=\left(u_d(0),u_d(1),\ldots,u_d(\tau-1)\right)^\T\in (\L^2([0,1];\KK^m))^\tau$ we have the representation 
$\mathcal{B}^\eqref{eq:WN-disc}_\tau u_d = \sum_{\sigma=0}^{\tau-1} A^{\tau-1-\sigma}B u_d(\sigma)= \left[\begin{smallmatrix}
A^{\tau-1}B & A^{\tau-2}B & \cdots & AB & B
\end{smallmatrix}\right]u_d =: Mu_d$.
Our claim is that $\ran \mathcal{B}_\tau^\eqref{eq:WN-disc}= \L^2\left([0,1]; \ran M \right)$, which would show the closedness of $\ran \mathcal{B}_\tau^\eqref{eq:WN-disc}$. We start by showing that $\ran \mathcal{B}_\tau^\eqref{eq:WN-disc}\subseteq \L^2\left([0,1]; \ran M \right)$. This follows from the estimation $\Nl\mathcal{B}^d_\tau u_d\Nr_{\L^2([0,1];\KK^n)}\leq \Nl \mathcal{B}^d_\tau\Nr \Nl u_d\Nr_{\left( \L^2([0,1];\KK^m) \right)^\tau}$. For the other direction, let $x_d\in \L^2([0,1];\ran M)$. Since $M: \KK^{\tau m} \to \ran M$ is surjective, we know that there exists $M^\mathfrak{r}:\ran M \to \KK^{\tau m}$ such that $ M M^\mathfrak{r}=I_{\ran M}$. We set $u_d:= M^\mathfrak{r}x_d$ and obtain $\mathcal{B}_\tau^d u_d =M u_d =x_d$, which implies 
$x_d\in \mathcal{B}_\tau^\eqref{eq:WN-disc}$.
\end{proof}

 The specific structure of $A,B,C$ and $D$ has not been used, which makes the previous statement applicable to arbitrary linear infinite-dimensional discrete-time systems with multiplication operators given by constant matrices.
Thanks to the equivalent representation of \eqref{eq:WN-P0} as an infinite-dimensional discrete-time system, we have the following theorem.
\begin{theorem}\label{thm:Control-WN-Fin}
    The following statements are equivalent:
    \begin{itemize}
        \item The infinite-dimensional system \eqref{eq:WN-P0} is exactly controllable.
        \item The infinite-dimensional system \eqref{eq:WN-P0} is approximately controllable.
        \item The finite-dimensional system \eqref{eq:fin-dim-syst} is controllable.
    \end{itemize}
\end{theorem}
\begin{proof}
    The proof follows from Proposition \ref{prop:controlConnectFinDis} and Lemma \ref{lem:ControlContDisc}.
    \end{proof}

We are now in position to prove our main result on controllability/observability.
\begin{proof}[Proof of Theorem \ref{thm:MainresContr}]
The proof concerning the controllability follows from Theorem \ref{thm:Control-WN-Fin} where the condition that \eqref{eq:fin-dim-syst} is controllable is replaced by the so-called \emph{Kalman's rank condition} and \emph{Hautus test} for finite-dimensional systems, see e.g.~\cite[Chap.~6]{HinD2026b2}. For the observability we use the expression for the dual system from Theorem \ref{prop:DualViaSysNode} and apply the previous statements for controllability to it. The statements for oberservability are obtained with the duality result from Proposition \ref{prop:DualInfDimConti}.
\end{proof}

\begin{remark}
Let $\mathcal A$ be defined as in \eqref{eq:HomoOper}. As we  assume that $K$ is invertible, the operator $\mathcal A$ generates a strongly continuous semigroup. Let us further assume that this semigroup is exponentially stable. 
In Russell and Weiss
\cite{RuWe94} it was shown that a necessary condition for exact
observability of  \eqref{eq:WN-P0} is the following
version of the Hautus test:
There exists $m>0$ such that for every $s\in\mathbb C_-$ and every $x \in D(\mathcal A)$:
\begin{align*}\tag{HT}\label{eq:Hautustest}
  \Nl(sI-\mathcal A)x\Nr_{\L^2([0,1],\KK^n)}^2 +|{\rm Re}\, s|\,\Nl K_y (\lambda_0{x})(0)+L_y (\lambda_0{x})(1)\Nr_{\KK^n}^2 \ge m|{\rm Re}\, s|^2\Nl x\Nr_{\L^2([0,1],\KK^n)}^2.
\end{align*}
Here $\mathbb C_-$ denotes the open left half plane. Further, in \cite{RuWe94} it was conjectured that the Hautus test \eqref{eq:Hautustest} is also sufficient for exact controllability. Twenty years later, the conjecture was disproved 
\cite{JaZw04}.  As the Hautus test
\eqref{eq:Hautustest} is sufficient for approximate observability of exponentially
stable systems \cite{RuWe94}, we conclude that the exponentially stable system    \eqref{eq:WN-P0} is exactly observable if and only if \eqref{eq:Hautustest} is satisfied.
\end{remark}

\section{Examples}\label{sec:Ex}
\subsection{A network of two transport equations}\label{sec:Ex-nice}
We consider the system \begin{align}\label{eq:System-Example-Nice}
\begin{split}
\frac{\partial x}{\partial t}(\xi,t) &= -\frac{\partial}{\partial \xi}(\lambda_0(\xi)x(\xi,t)), \quad t\geq 0,\,\, \xi\in [0,1],\quad x(\xi,0) = x_0(\xi),  \quad \xi\in [0,1], \\
\left(\begin{smallmatrix}0\\ 1\end{smallmatrix}\right)u(t) &= - \left(\begin{smallmatrix}
        0 & -1 \\ -1&0
    \end{smallmatrix}\right)\lambda_0(0)x(0,t) -  \left(\begin{smallmatrix}
        \frac{1}{2}& 0 \\ 0 &-\frac{1}{2}
    \end{smallmatrix}\right) \lambda_0(1)x(1,t), \quad  t\geq 0,\\
y(t) &=-  \left(\begin{smallmatrix}
        -1&0
    \end{smallmatrix}\right) \lambda_0(1) x(1,t),  \quad t\geq 0, \end{split}
\end{align}
with initial condition $x_0\in \L^2([0,1];\RR^{2})$, control $u(t)\in\RR$, observation $y(t)\in\RR$ and arbitrary $\lambda_0\in\linebreak \L^\infty([0,1],(0,\infty))$ such that $\lambda_0^{-1}\in\L^\infty([0,1],(0,\infty))$.
Observe that \eqref{eq:System-Example-Nice} can be cast into \eqref{eq:WN-P0} with \begin{align*}
    K=\left(\begin{smallmatrix}
        0 & -1 \\ -1&0
    \end{smallmatrix}\right),\quad L=\left(\begin{smallmatrix}
        \frac{1}{2}& 0 \\ 0 &-\frac{1}{2}
    \end{smallmatrix}\right), \quad K_y=0, \quad L_y= \left(\begin{smallmatrix}
        -1&0
    \end{smallmatrix}\right), \quad \lambda_0=\lambda_0,\quad  P_0\equiv0.
\end{align*}
Clearly the matrix $K$ is invertible, which, according to Proposition \ref{satz:well-posedness} implies that \eqref{eq:System-Example-Nice} is well-posed. According to \eqref{eq:Mat} the associated matrices $A, B, C$ and $D$ are given by $A=\left(\begin{smallmatrix}
        0& -\frac{1}{2}\\\frac{1}{2}&0
    \end{smallmatrix}\right), \, B=\left(\begin{smallmatrix}
        1\\0
    \end{smallmatrix}\right),\, C=\left(\begin{smallmatrix}
        1&0
    \end{smallmatrix}\right)$ and $D=0$, respectively. 
\subsubsection*{Stability and its robustness}
We observe that $\sigma(A)\subset \DD$, therefore the system is stable, see Proposition \ref{prop:MainresStab}. To calculate the stability radii, we consider the function $G(z)= C(z-A)^{-1} B+D =\frac{z}{z^2+\frac{1}{4}}$ on $\partial \DD$ and we compute the following quantity
\begin{align*}
    G(\mathrm{e}^{\mathrm{i}s})= \frac{\mathrm{e}^{\mathrm{i}s}}{\mathrm{e}^{\mathrm{i}2s}+\frac{1}{4}}=\frac{\mathrm{e}^{\mathrm{i}s}(\mathrm{e}^{-\mathrm{i}2s}+\frac{1}{4})}{(\mathrm{e}^{\mathrm{i}2s}+\frac{1}{4})(\mathrm{e}^{-\mathrm{i}2s}+\frac{1}{4})} =\frac{\frac{5}{4}\cos(s)-\mathrm{i}\frac{3}{4}\sin(s)}{\frac{17}{16}+\frac{1}{2}\cos(2s)}, \quad s\in[0,2\pi].
\end{align*}
In the case $\sin(s)\neq0$, we have that $\dist(\mathrm{Re}( G(e^{is})),\RR\mathrm{Im}( G(e^{is})) = 0$ 
and when $\sin(s)=0$, it holds that $\dist(\mathrm{Re}( G(e^{is})),\RR\mathrm{Im}( G(e^{is}))=\frac{4}{5}$.    
Thus, we conclude with Theorem \ref{thm:MainresRob} that the real stability radius is \begin{align*}
    r_\RR = \left[\max\{\Vert D\Vert, \max_{z\in\partial\DD}\dist (\mathrm{Re}(G(z)),\RR\mathrm{Im}(G(z)) )\}\right]^{-1} = \frac{5}{4}.
\end{align*} Now observe that \begin{align*}
    \max_{s\in[0,2\pi]} \left| G(\mathrm{e}^{\mathrm{i}s}) \right| ^2 = \max_{s\in[0,2\pi]} \left| \frac{1}{\mathrm{e}^{\mathrm{i}2s}+\frac{1}{4}}  \right|^2 =\max_{s\in[0,2\pi]} \frac{1}{\frac{17}{16}+\frac{1}{2}\cos(2s)}=\frac{16}{9},
\end{align*}
which, according to Theorem \ref{thm:MainresRob}, implies that the complex stability radius is $r_\CC = \frac{3}{4}$.
\subsubsection*{Controllability and observability} 
The matrices $\left[\begin{smallmatrix}
        B & AB
    \end{smallmatrix}\right] = \left(\begin{smallmatrix}
         1 & 0\\
         0& \frac{1}{2}
    \end{smallmatrix}\right)$ and $ \left[\begin{smallmatrix}
        C \\ CA
    \end{smallmatrix}\right] =\left(\begin{smallmatrix}
        1 &0\\ 0&-\frac{1}{2}
    \end{smallmatrix}\right)$ have full rank, hence \eqref{eq:System-Example-Nice} is exactly controllable and exactly observable by Theorem \ref{thm:MainresContr}.

\begin{remark}
    The system \eqref{eq:System-Example-Nice} with $\lambda_0=1$ is equivalent to \begin{align*}
    \frac{\partial^2w}{\partial t^2}(\xi,t) &=\frac{\partial^2w}{\partial \xi^2}(\xi,t),\quad t\geq 0, \xi\in [0,1], \quad w(\xi,0)=w_{0,1}(\xi), \frac{\partial w}{\partial t}(\xi,0) = w_{0,2}(\xi), \quad \xi\in [0,1],\\
    0&= -\frac{3}{4} \frac{\partial w}{\partial t}(0,t) +\frac{1}{4}\frac{\partial w}{\partial\xi}(0,t),\\
    u(t)&=\frac{1}{4}\frac{\partial w}{\partial t}(1,t) +\frac{3}{4}\frac{\partial w}{\partial\xi}(1,t),\quad t\geq 0,\\
    y(t)&=\frac{1}{2}\frac{\partial w}{\partial t}(0,t) +\frac{1}{2}\frac{\partial w}{\partial\xi}(0,t),\quad t\geq 0,
\end{align*}
with an initial condition $w_0 := (\begin{matrix}w_{0,1} & w_{0,2}\end{matrix})^\top\in\L^2([0,1],\RR^2)$. In the above wave equation $w$ is related to $x$ by $x_1(\xi,t)=\frac{1}{2}\left(\frac{\partial w}{\partial t}(1-\xi,t)+\frac{\partial w}{\partial \xi}(1-\xi,t)\right)$ and $x_2(\xi,t) = \frac{1}{2}\left(-\frac{\partial w}{\partial t}(\xi,t)+\frac{\partial w}{\partial \xi}(\xi,t)\right), t\geq 0, \xi\in [0,1]$, see \cite[Sec.~5]{HasA2025a3} for more details.
\end{remark}

\subsection{A co-current heat-exchanger}\label{sec:Ex-Heat-Exch}

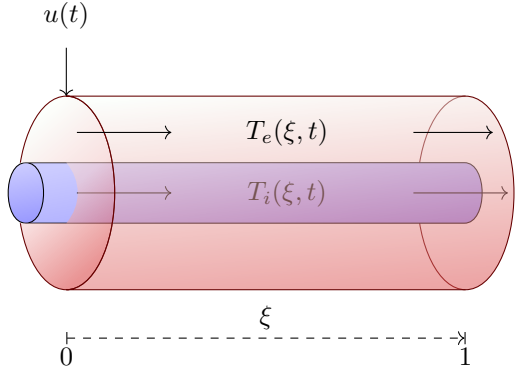
\begin{figure}
\begin{center}
%
\colorlet{Ecol}{orange!90!black}
\colorlet{EcolFL}{orange!80!black}
\colorlet{veccol}{green!45!black}
\colorlet{EFcol}{red!60!black}
\tikzstyle{charged}=[top color=blue!20,bottom color=blue!40,shading angle=10]
\tikzstyle{darkcharged}=[very thin,top color=blue!60,bottom color=blue!80,shading angle=10]
\tikzstyle{charge+}=[very thin,top color=red!80,bottom color=red!80!black,shading angle=-5]
\tikzstyle{charge-}=[very thin,top color=blue!50,bottom color=blue!70!white!90!black,shading angle=10]
\tikzstyle{gauss surf}=[red!40!black,top color=green!2,bottom color=red!80!black!70,shading angle=5,fill opacity=0.5]
\tikzstyle{gauss lid}=[gauss surf,middle color=red!80!black!20,shading angle=40,fill opacity=0.6]
\tikzstyle{gauss dark}=[red!50!black,fill=red!60!black!70,fill opacity=0.8]
\tikzstyle{gauss line}=[red!40!black]
\tikzstyle{gauss dashed line}=[red!60!black!80,dashed,line width=0.1]
\tikzstyle{EField}=[->,thick,Ecol]
\tikzstyle{vector}=[->,thick,veccol]
\tikzstyle{normalvec}=[->,thick,blue!80!black!80]
\tikzstyle{EFieldLine}=[thick,EcolFL,decoration={markings,
          mark=at position 0.5 with {\arrow{latex}}},
          postaction={decorate}]
\tikzstyle{measure}=[fill=white,midway,outer sep=2]
\def\L{2}
\def\W{0.2}
\def\N{4}
%
%
%
\begin{tikzpicture}[scale=4]
  \def\M{8}
  \def\R{0.4*\L}
  \def\g{0.2*\R}
  \def\G{0.4*\R}
  \def\a{0.33*\L}
  \coordinate (L)  at (-\a,0);
  \coordinate (R)  at (+\a,0);
  \coordinate (TL) at (-\a,\G);
  \coordinate (TR) at (+\a,\G);
  \coordinate (BL) at (-\a,-\G);
  \coordinate (BR) at (+\a,-\G);
  
  \draw[gauss line] (TR) arc (90:270:{\g} and {\G});
  
  \draw[charged] (-6.35*\L/16,-\W/2) --++(5.8*\L/8,0) to[out=0,in=0] ++ (0,\W) --++ (-5.8*\L/8,0) -- cycle;
  \draw[charged] (-6.35*\L/16,-\W/2) to[out=180,in=180] ++ (0,\W) to[out=0,in=0] cycle;
  \draw[->] (-5*\L/16,0) -- (-2.5*\L/16,0);
  \node[] at (0.07,0) {$T_i(\xi,t)$};
  \draw[->] (3.9*\L/16,0) -- (6.4*\L/16,0);
  
  \begin{scope}
    \clip (-\L/2,-0.5*\W)
      --++ (\L/2-\a,0) to[out=50,in=-50] ++(0,1.0*\W) --++ (-\L/2+\a,0) --++ (0,\G)
      --++ (\L-2*\a,0) --++ (0,{-2*(\G+\W)}) --++ (-\L+2*\a,0) -- cycle;
    \draw[gauss lid] (L) ellipse ({\g} and {\G});
  \end{scope}

  \draw[|->|,dashed] (-\a,-1.5*\G) -- (\a,-1.5*\G);
  \node[above] at (0,-1.5*\G) {$\xi$};
  \node[below] at (-\a,-1.5*\G) {$0$};
  \node[below] at (\a,-1.5*\G) {$1$};
  \draw[->] (-\a,1.5*\G) -- (-\a,1*\G);
  \node[above] at (-\a,1.6*\G) {$u(t)$};
  \draw[gauss surf]
    (BL) arc (-90:90:{\g} and {\G}) --
    (TR) arc (90:-90:{\g} and {\G}) -- cycle;
   \draw[->] (-5*\L/16,0.2) -- (-2.5*\L/16,0.2);
  \node[] at (0.07,0.2) {$T_e(\xi,t)$};
  \draw[->] (3.9*\L/16,0.2) -- (6*\L/16,0.2);
  
  
\end{tikzpicture}
\end{center}
\caption{Schematic profile view of a co-current heat-exchanger.}\label{fig:HE}
\end{figure}
We consider a co-current heat exchanger, which is a device consisting of two tubes of length 1, inserted into each other, see Figure \ref{fig:HE}. Each tube contains a fluid and both fluids are traveling in the same direction. Moreover, heat transfer is only happening between both tubes. By $\tilde{T}_i(\xi,t)$ and $\tilde{T}_e(\xi,t)$ we denote the temperature of the internal and external tube, respectively, at time $t\ge0$ and position $\xi\in[0,1]$. This system is described by the PDEs \begin{align*}
\frac{\partial {\tilde{T}_i}}{\partial t}(\xi,t) &= -v(\xi)\frac{\partial{\tilde{T}_i}}{\partial \xi}(\xi,t) + \alpha_i(\xi) \left(\tilde{T}_e(\xi,t)-\tilde{T}_i(\xi,t) \right), \quad \xi\in [0,1],\quad t\geq0, \\
\frac{\partial {\tilde{T}_e}}{\partial t}(\xi,t) &= -v(\xi)\frac{\partial{\tilde{T}_e}}{\partial \xi}(\xi,t) + \alpha_e(\xi) \left(\tilde{T}_i(\xi,t)-\tilde{T}_e(\xi,t) \right), \quad \xi\in [0,1],\quad t\geq0,
\end{align*} where $\alpha_i(\xi),\alpha_e(\xi)>0$, $\xi\in [0,1]$ are the heat transfer functions and $v(\xi)>0$, $\xi\in [0,1]$ the velocity of both fluids. Such a model may be found in e.g.~\cite[Sec.~2]{MaiA2010a} or \cite[Sec.~12.3]{OguB1994b}.
Since, $\alpha_i,\alpha_e$ and $v$ are natural quantities, we assume them sufficiently smooth, in particular we suppose that $\alpha_i, \alpha_e\in\L^\infty([0,1];(0,\infty))$ and $v\in\L^\infty([0,1];(0,\infty))$ such that $v^{-1}\in\L^\infty([0,1];(0,\infty))$. If we set $\tilde{T}(\xi,t):=\left(\begin{smallmatrix}
\tilde{T}_i(\xi,t) & \tilde{T}_e(\xi,t)
\end{smallmatrix}\right)^\T$, we obtain  \begin{align*}
\frac{\partial \tilde{T}}{\partial t}(\xi,t) &= -v(\xi)\frac{\partial\tilde{T}}{\partial \xi}(\xi,t) + \underbrace{\left(\begin{smallmatrix}-\alpha_i(\xi) & \alpha_i(\xi)\\ \alpha_e(\xi) & -\alpha_e(\xi)\end{smallmatrix}\right)}_{=:\tilde{P}_0(\xi)}\tilde{T}(\xi,t), \quad \xi\in[0,1],\quad  t\geq0, \quad
\tilde{T}(\xi,0)=\tilde{T}_0(\xi), \quad \xi\in [0,1],
\end{align*} where $\tilde{T}_0$ is the initial heat distribution. Furthermore, we assume that the temperature of the internal tube at $\xi=0$ vanishes. The input is the temperature of the external tube at $\xi=0$ and the output is the difference of temperature in the outer tube between the outlet and the inlet, leading to the following conditions \begin{align*}
\left(\begin{smallmatrix}0\\ 1\end{smallmatrix}\right)u(t) &= \tilde{T}(0,t),\quad t\geq 0,\\
y(t) &= \left(\begin{smallmatrix}0 & -1\end{smallmatrix}\right)\tilde{T}(0,t) + \left(\begin{smallmatrix}0 & 1\end{smallmatrix}\right)\tilde{T}(1,t),\quad t\geq 0.
\end{align*}
Now we define ${T}(\xi,t):= \frac{1}{v(\xi)}\tilde{T}(\xi,t)$, $\xi\in[0,1], t\geq0$. Hence, 
\begin{align}\label{eq:HeatExch}
\begin{split}
    \frac{\partial T }{\partial t}(\xi,t) &= -\frac{\partial}{\partial \xi}(v(\xi)T (\xi,t)) + \tilde{P}_0(\xi)T (\xi,t), \quad \xi\in[0,1],\quad  t\geq0, \quad T (\xi,0)=v(\xi)^{-1}\tilde{T}_0(\xi), \quad \xi\in[0,1],\\
\left(\begin{smallmatrix}0\\ 1\end{smallmatrix}\right)u(t) &= v(0)T (0,t), \quad  t\geq0, \\
y(t) &= (\begin{smallmatrix}0 & -1\end{smallmatrix})v(0)T (0,t) + (\begin{smallmatrix}0 & 1\end{smallmatrix})v(1)T (1,t), \quad t\geq0.
\end{split}
\end{align}
Thus, \eqref{eq:HeatExch} has the form \eqref{eq:WN-P0} with $\lambda_0(\cdot) = v(\cdot), P_0(\cdot)=\frac{1}{v(\cdot)}\tilde{P}_0(\cdot), K = -I, L = 0, K_y = \left(\begin{smallmatrix}0 & 1\end{smallmatrix}\right), L_y = -\left(\begin{smallmatrix}0 & 1\end{smallmatrix}\right)$. The invertibility of $K$ implies that \eqref{eq:HeatExch} is well-posed, see Proposition \ref{satz:well-posedness}. The solution to $P'(\xi) = v(\xi)^{-1}\tilde{P}_0(\xi)P(\xi), P(0) = I$ is given by \begin{align*}
P(\xi) = \left(\begin{smallmatrix} 1-\int_0^\xi e^{h(\eta)}\frac{\alpha_i(\eta)}{v(\eta)}\d\eta &  1 - e^{h(\xi)} - \int_0^\xi e^{h(\eta)}\frac{\alpha_e(\eta)}{v(\eta)}\d\eta\\
 1 - e^{h(\xi)}-\int_0^\xi e^{h(\eta)}\frac{\alpha_i(\eta)}{v(\eta)}\d\eta &  1 - \int_0^\xi e^{h(\eta)}\frac{\alpha_e(\eta)}{v(\eta)}\d\eta\end{smallmatrix}\right)
\end{align*}
with $h:[0,1]\to \mathbb{R}^-, h(\eta) := -\int_0^\eta v(\zeta)^{-1}\left(\alpha_i(\zeta)+\alpha_e(\zeta)\right)\d\zeta$, see \cite[Sec.~4]{HasA2025a}. Therefore the matrices in \eqref{eq:Mat} are given by \begin{align*}
A &= 0, \qquad B = \left(\begin{smallmatrix}0\\ 1\end{smallmatrix}\right),\qquad C = \left(\begin{smallmatrix} 1-e^{h(1)} -\int_0^1 e^{h(\eta)}\frac{\alpha_i(\eta)}{v(\eta)}\d\eta &\hspace{0.7cm}  1 - \int_0^1 e^{h(\eta)}\frac{\alpha_e(\eta)}{v(\eta)}\d\eta\end{smallmatrix}\right)=:\left(\begin{smallmatrix}
    C_1 &C_2
\end{smallmatrix}\right),\quad D = -1.
\end{align*}

\subsubsection*{Stability and its robustness}
It holds that $\sigma(A)\subset\DD$, which implies that \eqref{eq:HeatExch} with $u\equiv 0$ is stable, see Proposition \ref{prop:MainresStab}. According to Theorem \ref{thm:MainresRob} we consider the function $G(z)=C(z-A)^{-1} B+D = \frac{C_2}{z}-1$. The complex stability radius is given by $\left(\max_{z\in\partial\DD} \left| \frac{C_2}{z}-1 \right|\right)^{-1} =\left(\max_{z\in\partial\DD} \left| C_2 \bar{z}-1 \right|\right)^{-1} =\left(|C_2|+1\right)^{-1}$, see Theorem \ref{thm:MainresRob}. For $z\in\partial\DD$ it holds that $G(z)= C_2 \bar{z}-1 = C_2 \Re z-1- \mathrm{i}C_2\Im z$. If $\Im z\neq 0$ and $C_2\neq0$ we have $\dist \left( \Re G(z), \RR \Im G(z) \right) = \dist \left( C_2\Re z-1, -\RR C_2\Im z \right)=0$. For $\Im z \neq 0$ and $C_2=0$ it holds that $\dist \left( \Re G(z), \RR \Im G(z) \right) =1$.
If $\Im z=0 $ we get $\max_{\Im z=0, z\in\partial\DD} \dist \left( \Re G(z), \RR \Im G(z) \right)= 1+|C_2|$. Noting that $|D|=1$ implies that the real stability radius is also given by $\frac{1}{1+|C_2|}$, see Theorem \ref{thm:MainresRob}.

\subsubsection*{Controllability and observability}
Both matrices $\left[\begin{smallmatrix}
B & A B
\end{smallmatrix}\right] = \left(\begin{smallmatrix}
0 & 0 \\ 1 &0
\end{smallmatrix}\right)$ and $
\left[\begin{smallmatrix}
C \\ CA
\end{smallmatrix}\right] = \left(\begin{smallmatrix}C_1&C_2\\ 0 & 0\end{smallmatrix}\right)$ have no full rank. Hence, Theorem \ref{thm:MainresContr} implies that \eqref{eq:HeatExch} is neither exactly nor approximately controllable and neither exactly nor approximately observable.

\section{Perspectives}\label{sec:Persp}
Extending the proposed approach to systems where $\lambda_0$ is replaced by a diagonal matrix with different entries would be worth investigating in further research. This would notably help in covering a much broader class of applications like the Timoshenko beam \cite{JacB2012b} or counter-current heat exchangers. This extension would be the first step towards the extension of our results to the class of so-called first-order port-Hamiltonian systems, see e.g.~\cite{JacB2012b}. Further research could also aim at considering other perturbations than output feedback and to study the robustness under those perturbations. This perspective could open the door to the design of robust stabilizing feedback for \eqref{eq:WN-P0}. Another extension of this work could be dedicated to the study of the transfer function of \eqref{eq:WN-P0}, and more specifically a study of its properties in terms of belonging to some algebras of transfer functions.

\section*{Acknowledgments}
This research was conducted with the financial supports of the F.R.S-FNRS (Belgium) and the Deutsche Forschungsgemeinschaft (DFG, German Research Foundation). A.~Hastir is supported by the FNRS, Grant CR 40010909. He has also been supported by the DFG - Project-ID 532208976. All the authors acknowledge funding by the DFG, Germany
- Project-ID 531152215 - CRC 1701.

\section*{AI disclosure}
OpenAI ChatGPT was used to assist in constructing the change of variables presented in Remark \ref{rem:dual}. The authors reviewed all the AI-assisted parts independently.

\bibliographystyle{abbrvnat}
\bibliography{0Bibliothek.bib}
\end{document}